\documentclass[twoside,journal]{IEEEtran}
\usepackage{amsmath,graphicx}
\usepackage{algorithm}
\usepackage{algorithmic}
\usepackage{enumerate}
\usepackage{amssymb}
\usepackage{amsmath}
\usepackage{amsthm}
\usepackage{epstopdf}
\usepackage{multirow}
\usepackage{pifont,color}% http://ctan.org/pkg/pifont
\usepackage[dvipsnames]{xcolor}

\def\MatrixFont{\bf}
\def\VectorFont{\bf}

\newcommand{\bm}[1]{\mbox{\boldmath{$#1$}}}
\newcommand{\mA}{{\MatrixFont A}}
\newcommand{\mB}{{\MatrixFont B}}
\newcommand{\mC}{{\MatrixFont C}}
\newcommand{\mD}{{\MatrixFont D}}

\newcommand{\mI}{{\MatrixFont I}}

\newcommand{\mS}{{\MatrixFont S}}

\newcommand{\mW}{{\MatrixFont W}}

\newcommand{\va}{{\VectorFont a}}

\newcommand{\vp}{{\VectorFont p}}

\newcommand{\vs}{{\VectorFont s}}

\newcommand{\vw}{{\VectorFont w}}
\newcommand{\vx}{{\VectorFont x}}
\newcommand{\vy}{{\VectorFont y}}
\newcommand{\vz}{{\VectorFont z}}

\newcommand{\vxi}{{\boldsymbol{\xi}}}
\newcommand{\vlambda}{{\boldsymbol{\lambda}}}

\newcommand{\vmu}{{\boldsymbol{\mu}}}
\newcommand{\vzeta}{{\boldsymbol{\zeta}}}
\newcommand{\cmark}{\ding{51}}%
\newcommand{\xmark}{\ding{55}}
\newtheorem{theorem}{Theorem}
\newtheorem{lemma}{Lemma}
\newtheorem{remark}{Remark}

\newtheorem{assumption}{Assumption}

\newtheorem{corollary}{Corollary}
\newcommand\Ec{\ensuremath{\mathcal{E}}}
\newcommand\Gc{\ensuremath{\mathcal{G}}}
\newcommand\Nc{\ensuremath{\mathcal{N}}}

\newcommand\Ab{\ensuremath{{\bf A}}}

\newcommand\bb{\ensuremath{{\bm b}}}

\newcommand\Ub{\ensuremath{{\bf U}}}

\newcommand\Wb{\ensuremath{{\bf W}}}

\newcommand\zerob{\ensuremath{{\bm 0}}}
\newcommand\oneb{\ensuremath{{\bm 1}}}

\newcommand\xib{\ensuremath{{\bm \xi}}}

\newcommand\Psib{\ensuremath{{\bf \Psi}}}

\newcommand\E{\ensuremath{{\mathbb{E}}}}

\ifCLASSINFOpdf
\else
\fi
\begin{document}
%
% paper title
% Titles are generally capitalized except for words such as a, an, and, as,
% at, but, by, for, in, nor, of, on, or, the, to and up, which are usually
% not capitalized unless they are the first or last word of the title.
% Linebreaks \\ can be used within to get better formatting as desired.
% Do not put math or special symbols in the title.
\title{Distributed Stochastic Consensus Optimization with Momentum for Nonconvex
 Nonsmooth Problems}
%
%
% author names and IEEE memberships
% note positions of commas and nonbreaking spaces ( ~ ) LaTeX will not break
% a structure at a ~ so this keeps an author's name from being broken across
% two lines.
% use \thanks{} to gain access to the first footnote area
% a separate \thanks must be used for each paragraph as LaTeX2e's \thanks
% was not built to handle multiple paragraphs
%

\author{Zhiguo~Wang,
        Jiawei~Zhang,
        Tsung-Hui~Chang, Jian Li and Zhi-Quan~Luo% <-this % stops a space
%\thanks{~}% <-this % stops a space
\thanks{ Zhiguo Wang is with College of Mathematics, Sichuan
University, Chengdu, Sichuan 610064, China (e-mail: wangzhiguo@scu.edu.cn). This work was done in the Chinese University of Hong Kong, Shenzhen.}% <-this % stops a space
\thanks{ Jiawei Zhang, Tsung-Hui Chang, and Zhi-Quan (Tom) Luo are with the Chinese University of Hong Kong, Shenzhen 518172, China
and also with Shenzhen Research Institute of Big Data , Shenzhen, Guangdong Province 518172, China (e-mail: jiaweizhang2@link.cuhk.edu.cn;
tsunghui.chang@ieee.org; luozq@cuhk.edu.cn). Corresponding
author: Zhi-Quan (Tom) Luo.}
\thanks{ Jian Li is with Department of Electrical and Computer Engineering, University of Florida (e-mail: li@dsp.ufl.edu).}
}

\maketitle

% As a general rule, do not put math, special symbols or citations
% in the abstract or keywords.
\begin{abstract}
While many distributed optimization algorithms have been proposed for solving smooth or convex problems over the networks, few of them can handle non-convex and non-smooth problems.
Based on a proximal primal-dual approach, this paper presents a new (stochastic) distributed algorithm with Nesterov momentum for accelerated optimization of non-convex and non-smooth problems. Theoretically, we show that the proposed algorithm can achieve an $\epsilon$-stationary solution under a constant step size with $\mathcal{O}(1/\epsilon^2)$ computation complexity  and $\mathcal{O}(1/\epsilon)$ communication complexity.
When compared to the existing gradient tracking based methods, the proposed algorithm has the same order of computation complexity but lower order of communication complexity.  To the best of our knowledge, the presented result is the  first stochastic algorithm with the $\mathcal{O}(1/\epsilon)$ communication complexity for non-convex and non-smooth problems.
Numerical experiments for a distributed non-convex regression problem and a deep neural network based classification problem are presented to illustrate the effectiveness of the proposed algorithms.
\end{abstract}

% Note that keywords are not normally used for peerreview papers.
\begin{IEEEkeywords}
Distributed optimization, stochastic optimization, momentum, non-convex and non-smooth optimization.
\end{IEEEkeywords}

% For peer review papers, you can put extra information on the cover
% page as needed:
% \ifCLASSOPTIONpeerreview
% \begin{center} \bfseries EDICS Category: 3-BBND \end{center}
% \fi
%
% For peerreview papers, this IEEEtran command inserts a page break and
% creates the second title. It will be ignored for other modes.
\IEEEpeerreviewmaketitle

\vspace{-0.3cm}
\section{Introduction}
\vspace{-0.0cm}
\label{sec:intro}
Recently, motivated by large-scale machine learning \cite{boyd2011distributed} and mobile edge computing \cite{mach2017mobile}, many
signal processing applications involve handling very large datasets \cite{vlaski2019distributed} that are processed over networks with distributed memories and processors.
Such signal processing and machine learning problems are usually formulated as a multi-agent distributed optimization problem \cite{chang2015multi}. In particular, many of the applications can be formulated as the following finite sum problem
\begin{align}
\label{finite sum prob}\min_{x} ~\sum_{i=1}^N\Big(f_i(x)+r_i(x)\Big),
\end{align}
where $N$ is the number of agents, $x\in \mathbb{R}^n$ contains the model parameters to be learned,
$f_i(x): \mathbb{R}^n\rightarrow\mathbb{R}$ is a closed and smooth (possibly nonconvex) loss function,
and $r_i(x)$ is a convex and possibly non-smooth regularization term.
%In addition, we consider two classical settings for the local cost functions
Depending on how the data are acquired, there are two scenarios for problem \eqref{finite sum prob} \cite{chang2020distributed}.
\begin{itemize}
	
\item
Offline/Batch learning: the agents are assumed to have the complete local dataset. Specifically, the local cost functions can be written as
\begin{align}
f_i(x)=\frac{1}{m}\sum_{j=1}^{m}f_i^j(x),~ i=1,\ldots,N,
\end{align}
where
$f_i^j(x)$ is the cost for the $j$-th data sample at the $i$-th agent, and $m$ is the total number of local samples.
%In this setting, the data oracle can return a full gradient, then it is effective to develop deterministic first-order methods.
When $m$ is not large, each agent $i$ may compute the full gradient of $f_i(x)$ for deterministic parameter optimization.

  \item
  Online/Streaming learning: when the data samples follow certain statistical distribution and
  are acquired by the agents in an online/streaming fashion, one can define
$f_i(x)$  as the following expected cost
\begin{align}\label{online cost}
f_i(x)=\mathbb{E}_{\xi\sim \mathcal{B}_i}[f_i(x,\xi)], i=1,\ldots,N,
\end{align}
where $\mathcal{B}_i$ denotes the data distribution at agent $i$, and $f_i(x,\xi)$ is the cost function of a random data sample $\xi$.
Under the online setting, only a stochastic estimate $G_i(x,\xi)$ for $\nabla f_i(x)$ can be obtained by the agent and stochastic optimization methods can be used.
Note that if the agent is not able to compute the full gradient in the batch setting, a stochastic gradient estimate by mini-batch data samples can be obtained and the problem is solved in a similar fashion by stochastic optimization.
%This setting is typical in processing large datasets where computing the full gradient is impossible, then we should design stochastic first-order methods.
\end{itemize}
These two settings for local cost functions are popularly used in many machine learning models including deep learning and empirical risk minimization problems \cite{chang2020distributed}.
For both scenarios, many distributed optimization methods have been developed for solving problems \eqref{finite sum prob}.

Specifically, for batch learning and under convex or strongly convex assumptions, algorithms such as the distributed subgradient method \cite{nedic2009distributed}, EXTRA
\cite{shi2015extra}, PG-EXTRA \cite{shi2015proximal} and primal-dual based methods including the alternating direction method of multipliers (ADMM) \cite{boyd2011distributed,chang2015multi,shi2014linear} and the UDA in \cite{alghunaim2019decentralized} are proposed. For non-convex problems, the authors in \cite{zeng2018nonconvex} studied the convergence of proximal decentralized gradient descent (DGD) method with a diminishing step size.
Based on the successive convex approximation (SCA) technique and the gradient tracking (GT) method, the authors in \cite{di2016next} proposed a network successive convex approximation (NEXT) algorithm for \eqref{finite sum prob}, and it is extended to more general scenarios with time varying networks and stronger convergence analysis results \cite{scutari2018parallel,scutari2019distributed}.
%Recently, in \cite{vlaski2019linear,vlaski2020second}, they consider the  second-order stationary points for decentralized non-convex optimization \eqref{finite sum prob}.
%The above methods belong to a family of gradient tracking methods.
In \cite{hong2017prox}, based on an inexact augmented Lagrange method, a proximal primal-dual algorithm (Prox-PDA) is developed for \eqref{finite sum prob} with smooth and non-convex $f_i(x)$ and without $r_i(x)$. A near-optimal algorithm xFilter is further proposed in \cite{sun18optimal} that can achieve the computation complexity lower bound of first-order distributed optimization algorithms.
To handle non-convex and non-smooth problems with polyhedral constraints, the authors of {\cite{zhang2018proximal,zhang2020}} proposed a proximal augmented Lagrangian (AL) method for solving \eqref{finite sum prob}
by introducing a proximal variable and an exponential averaging scheme.
%
%that can handle non-smooth and non-convex problems with the same setting as \eqref{finite sum prob}. But the above methods need to know the full gradient of local cost, then it can only deal with the offline setting with box constraint set rather than a general affine constraint set.

For streaming learning, the stochastic proximal gradient consensus method based on ADMM is proposed in \cite{hong2017stochastic} to solve \eqref{finite sum prob} with convex objective functions.
For non-convex problems, the decentralized parallel stochastic gradient descent (D-PSGD) \cite{lian2017can} is applied to
\eqref{finite sum prob} (without $r_i(x)$) for training large-scale neural networks, and the convergence rate is analyzed.
The analysis of D-PSGD relies on an assumption that $\frac{1}{N}\sum_{i=1}^N||\nabla f_i(x)- \nabla f(x)||^2$ is bounded, which implies that the variance of data distributions across the agents should be controlled.
In \cite{tang2018d}, the authors proposed an improved D-PSGD algorithm, called $D^2$,
which removes such assumption and is less sensitive to the data variance across agents.
However, $D^2$ requires a restrictive assumption on the eigenvalue of the mixing matrix.
This assumption is relaxed by the GNSD algorithm in \cite{lu2019gnsd}, which essentially is a stochastic counterpart of the GT algorithm in \cite{scutari2019distributed}.
%Indeed, \cite{tang2018d,lu2019gnsd} are also gradient tracking methods and their variant is popular used to solve nonconvex smooth optimization problems (see \cite{chang2020distributed,jiang2017collaborative,assran2018stochastic}).
We should emphasize here that the algorithms in \cite{lian2017can,tang2018d,lu2019gnsd} can only handle smooth problems without constraints and regularization terms.
%However, if the objective function has nonsmooth
% term, such as $\ell_1$ norm, then the above methods cannot be applied. Although paper \cite{huang2018mini} also considers the nonsmooth nonconvex optimization with linear constraint $Ax=b$, it assumes $A$ is a full row
%rank, which is limited in  distributed stochastic non-convex problem \eqref{finite sum prob}.
The work
\cite{bianchi2012convergence} proposed a multi-agent projected stochastic gradient decent (PSGD) algorithm for \eqref{finite sum prob} but $r_i(x)$ is limited to the indicator function of compact convex sets. Besides, there is no convergence rate analysis in \cite{bianchi2012convergence}.

In this paper, we develop a new distributed stochastic optimization algorithm for the non-convex and non-smooth problem \eqref{finite sum prob}. The proposed algorithm is inspired by the proximal AL framework in \cite{zhang2018proximal} and has three new features.
First, the proposed algorithm is a stochastic distributed algorithm that can be used either for streaming/online learning or batch/offline learning with mini-batch stochastic gradients.
Second, the proposed algorithm can handle problem \eqref{finite sum prob} with non-smooth terms that have a polyhedral epigraph, which is more general than  \cite{zhang2018proximal,zhang2020}.
Third, the proposed algorithm incorporates the Nesterov momentum technique for fast convergence.
The Nesterov momentum technique has been applied for accelerating the convergence of
distributed optimization. For example, in \cite{jakovetic2014fast,li2018sharp}, the distributed gradient descent methods with the Nesterov momentum are proposed, and are shown to achieve the optimal iteration complexity for convex problems.
In practice, since SGD with momentum often can converge faster, it is also commonly used to train deep neural networks \cite{yan2018unified,yu2019linear}.
We note that \cite{jakovetic2014fast,li2018sharp,yan2018unified,yu2019linear} are for smooth problems. To the best of our knowledge, the Nesterov momentum technique has not been used for distributed non-convex and non-smooth optimization. %It seems that the distributed primal dual method with momentum will have greater potential in theory and practice.

%This paper considers the distributed stochastic non-convex and non-smooth optimization problem
%over a connected network.
Our contributions are summarized as follows.
\begin{itemize}
  \item We propose a new stochastic proximal primal dual algorithm with momentum (SPPDM) for non-convex and non-smooth problem \eqref{finite sum prob} under the online/streaming setting.
 For the offline/batch setting where the full gradients of the local cost functions are available, SPPDM reduces to a deterministic algorithm, named the PPDM algorithm.
  \item We show that the proposed SPPDM and PPDM can achieve an $\epsilon$-stationary solution of \eqref{finite sum prob} under a constant
  step size with computation complexities of $\mathcal{O}(1/\epsilon^2)$ and $\mathcal{O}(1/\epsilon)$, respectively, while both have a communication complexity of $\mathcal{O}(1/\epsilon)$.
        %To our knowledge, this is the first time that a distributed momentum proximal primal dual method for non-convex stochastic optimization is proven to enjoy the sublinear convergence.
        The convergence analysis neither requires assumption on the boundedness of $\frac{1}{N}\sum_{i=1}^N\|\nabla f_i(x)- \nabla f(x)\|^2$ nor on the eigenvalues of the mixing matrix.

  \item As shown in Table \ref{Tab:01}, the proposed SPPDM/PPDM algorithms have the same order of computation complexity as the existing methods and lower order of communication complexity when compared with the existing GT based methods.
  \item Numerical experiments for a distributed non-convex regression problem and a deep neural network (DNN) based classification problem show that the proposed algorithms outperforms
	the existing methods.
\end{itemize}
\begin{table*}[t]
\caption{Comparisons of different algorithms}
\label{Tab:01}
\center
\begin{tabular}[l]{ccccccc}
\hline
 Algorithm & objective function & gradient
         & stepsize
         & momentum & computational & communication \\
\hline
  D-PSGD \cite{lian2017can} &$f(\vx)$ & stochastic & decreasing & \xmark & $\mathcal{O}(\frac{N}{\epsilon^2})$  &$\mathcal{O}(\frac{1}{\epsilon^2})$ \\[1mm]
  D$^2$ \cite{tang2018d}& $f(\vx)$ & stochastic  & decreasing & \xmark & $\mathcal{O}(\frac{N}{\epsilon^2})$  &$\mathcal{O}(\frac{1}{\epsilon^2})$\\[1mm]
  GNSD \cite{lu2019gnsd}& $f(\vx)$ & stochastic  & decreasing & \xmark & $\mathcal{O}(\frac{N}{\epsilon^2})$  &$\mathcal{O}(\frac{1}{\epsilon^2})$\\[1mm]
  PR-SGD-M \cite{yu2019linear}&$f(\vx)$ & stochastic & decreasing & \cmark & $\mathcal{O}(\frac{N}{\epsilon^2})$  &$\mathcal{O}(\frac{1}{\epsilon^2})$\\[1mm]
    PSGD \cite{bianchi2012convergence}&$f(\vx)+r(\vx)$ & stochastic & decreasing & \xmark & \xmark & \xmark \\[1mm]
  STOC-ADMM \cite{huang2018mini}&$f(\vx)+r(\vx)$ & stochastic & fixed & \xmark & $\mathcal{O}(\frac{N}{\epsilon^2})$  &$\mathcal{O}(\frac{1}{\epsilon})$\\[1mm]
  Prox-PDA \cite{hong2017prox}&$f(\vx)$ & full  & fixed & \xmark & $\mathcal{O}(\frac{mN}{\epsilon})$
  &$\mathcal{O}(\frac{1}{\epsilon})$  \\[1mm]
    Prox-DGD \cite{zeng2018nonconvex}&$f(\vx)+r(\vx)$ & full  & decreasing & \xmark & \xmark& \xmark \\[1mm]
    Prox-ADMM \cite{zhang2018proximal}&$f(\vx)+r(\vx)$ & full  & fixed & \xmark &$\mathcal{O}(\frac{mN}{\epsilon})$&$\mathcal{O}(\frac{1}{\epsilon})$\\[1mm] \hline
  \multirow{2}*{Proposed}&\multirow{2}*{$f(\vx)+r(\vx)$} & full  & fixed & \cmark &$\mathcal{O}(\frac{mN}{\epsilon})$&$\mathcal{O}(\frac{1}{\epsilon})$\\[1mm]
   \cline{3-7}
   ~&~ & stochastic  & fixed & \cmark &$\mathcal{O}(\frac{N}{\epsilon^2})$&$\mathcal{O}(\frac{1}{\epsilon})$\\
\hline
\end{tabular}
\end{table*}
%\subsection{Notation}

 {\bf Notation: } We denote $\mI_n$ as the $n$ by $n$ identity matrix and $\oneb$ as the all-one vector, i.e.,  $\oneb=[1,\ldots,1]^{\top}$.
$\langle \va, \bb \rangle$ represents the inner product of vectors $\va$ and $\bb$, $\|\va\|$ is the Euclidean norm of vector $\va$ and $\|\va\|_1$ is the $\ell_1$-norm of vector $\va$; $\otimes$ denotes the Kronecker product.
For a matrix $\mA$, $\sigma_A>0$ denotes its largest singular value.
${\rm diag}\{a_1,\ldots,a_N\}$ denotes a diagonal matrix with $a_1,\ldots,a_N$ being the diagonal entries while ${\rm diag}\{\mA_1,\ldots,\mA_N\}$ denotes a block diagonal matrices with each $\mA_i$ being the $i$th block diagonal matrix. $[\mA]_{ij}$ represents the element of $\mA$ in the $i$th row and $j$th column.

For problem \eqref{finite sum prob}, we denote $\vx=[x_1^{\top},\ldots,x_N^{\top}]^{\top}\in \mathbb{R}^{N n}$, $f(\vx)=\sum_{i=1}^Nf_i(x_i)$, and $r(\vx)=\sum_{i=1}^N r_i(x_i)$. The gradient of $f(\cdot)$ at $\vx$ is denoted by
$$
\nabla f(\vx)= [(\nabla f_1(x_1))^{\top},\ldots,(\nabla f_N(x_N))^{\top}]^{\top},
$$
where $\nabla f_i(x_i)$ is the gradient of $f_i$ at $x_i$. In the online/streaming setting, we denote the stochastic gradient estimates of agents as
$$
G(\vx,\vxi)= [(G_i(x_1,\xi_1))^{\top},\ldots,(G_N(x_N,\xi_N))^{\top}]^{\top},
$$ where $\vxi = [\xi_1^{\top},,\ldots,\xi_N^{\top}]$.
%Let $|\mathcal{I}|$ be the size of mini-batch $\mathcal{I}$.
Lastly, we define the following proximal operator of
$r_i$
\begin{align*}
\textmd{prox}_{r_i}^\alpha(x) = \arg\min_{u}\frac{\alpha}{2}\|x-u\|^2+r_i(u),
\end{align*} where $\alpha$ is a parameter.
%It is assume that $r_i$ allows $\textmd{prox}_{r_i}^\alpha(x)$ to yield a simple solution, such as when $r_i(x)$ is the $\ell_1$-norm $\|x\|_1$ or the indicator function
%of simple constraints. In addition, $\|\cdot\|$ is $\ell_2$-norm.

%In this paper,

 {\bf Synopsis:} In Section \ref{sec: alg develop}, the proposed SPPDM and PPDM  algorithms are presented and their connections with existing methods are discussed. Based on an inexact stochastic primal-dual framework, it is shown how the SPPDM and PPDM  algorithms are devised.
Section \ref{sec: conv analysis} presents the theoretical results of the convergence conditions and convergence rate of the SPPDM and PPDM algorithms. The performance of the SPPDM and PPDM algorithms are illustrated in Section \ref{sec: simulation}. Lastly, the conclusion is given in Section \ref{sec:conclusion}.

\section{Algorithm Development}\label{sec: alg develop}

\subsection{Network Model and Consensus Formulation}
Let us denote the multi-agent network as a graph $\mathcal{G}$, which contains a node set $V:=\{1,\ldots,N\}$ and an edge set $\mathcal{E}$ with cardinality $|\mathcal{E}|$. For each agent $i$, it has neighboring agents in the subset
$\mathcal{N}_i:=\{j\in V|(i,j)\in \mathcal{E}\}$ with size $d_i\geq 1$. It is assumed that each agent $i$ can communicate with  its
neighborhood $\mathcal{N}_i$.
We also assume that the graph $\mathcal{G}$ is undirected and is connected in the sense that for
any of two agents in the network there is a path connecting them through the edge links. Thus, problem \eqref{finite sum prob} can be equivalently written as
\begin{subequations}\label{eqn: dec}
\begin{align}
\label{equ_2}\min_{\substack{x_i\\ i=1,\ldots,N}}&~ \sum_{i=1}^N\Big(f_i(x_i)+r_i(x_i)\Big)\\
\label{equ_3} \textmd{s.t.} &~ x_i = x_j, ~\forall (i,j)\in \mathcal{E} .
\end{align}
\end{subequations}

Let us introduce the incidence matrix $\tilde{\mA}\in \mathbb{R}^{|\mathcal{E}|\times n}$ which has $\tilde{\mA}(\ell,i)=1$ and $\tilde{\mA}(\ell,j)=-1$ if $(i,j)\in \mathcal{E}$ with $j>i$, and zero otherwise, for $\ell=1,\ldots,|\Ec|$.
Define the extended incidence matrix as $\mA:=\tilde{\mA}\otimes \mI_{n}$. %, where $\otimes$ denotes the Kronecker product.
Then \eqref{eqn: dec} is equivalent to
\begin{subequations}\label{eqn: dec_compact}
	\begin{align}
	\label{equ_2_1a}\min_{\vx}&~ f(\vx)+r(\vx)\\
	\label{equ_3_2a} \textmd{s.t.} &~ \mA\vx={\mathbf 0}.
	\end{align}
\end{subequations}
%{\red Note that if $r(\vx)=0$, optimization \eqref{eqn: dec_compact} reduces the problem considered in \cite{zhang2020}. In this paper, $r(\vx)$ can be a general nonsmooth term such as $\ell_1$ norm.}
\subsection{Proposed SPPDM and PPDM Algorithm}\label{subsec: sppdm}
In this section, we present the proposed SPPDM algorithm for solving \eqref{eqn: dec_compact} under the online/streaming setting in \eqref{online cost}.
The algorithm steps are outlined in Algorithm \ref{alg_ADMM2}.
\begin{algorithm}
\begin{algorithmic}
\STATE {\bf Given} parameters $\alpha,\beta,\gamma,c,\kappa,\eta_k$ and initial values of $x_i^0$, $i=1,\ldots,N.$
Let
\begin{align}
\psi_i=\gamma+2cd_i+\kappa
\end{align} and set $s_i^0=x_i^0$, $i=1,\ldots,N.$ Do
\vspace{-0.4cm}
\STATE\begin{align}
\nonumber x_i^{\frac{1}{2}} &=(\gamma+cd_i+\kappa)\frac{x_i^0}{\psi_i}+\frac{c}{\psi_i}\sum_{j\in \mathcal{N}_i}x_j^0-\frac{1}{\psi_i}\nabla f_i(x_i^0),\\
\nonumber x_i^1 &= \textmd{prox}_{r_i}^{\alpha_i}\left(x_i^{\frac{1}{2}}\right),~i=1,\ldots,N.
\end{align}%for $i=1,\ldots,N$.
\FOR{communication round $k=1,2,\ldots$}
\FOR{agent $i=1,2,\ldots,N$ (in parallel)}
\vspace{-0.4cm}
\STATE \begin{align}
\label{mom}s_i^{k}&=x_i^k+\eta_k(x_i^k-x_i^{k-1}),\\
\label{xk12}x_i^{k+\frac{1}{2}}&=x_i^{k-1+\frac{1}{2}}+\frac{d_i}{\psi_i}((c-\alpha)x_i^k-c x_i^{k-1}),\\
\nonumber &~~~~+\frac{1}{\psi_i}\sum_{j\in \mathcal{N}_i}((c+\alpha)x_j^{k}-c x_j^{k-1})\\
\nonumber &~~~~+\frac{1}{\psi_i}\big(\gamma(s_i^k-s_i^{k-1})+\kappa(z_i^k-z_i^{k-1})\big)\\
\nonumber &~~~~-\frac{1}{\psi_i|\mathcal{I}|} \sum_{j=1}^{|\mathcal{I}|} (G_i(s_i^k,\xi_{ij}^k)- G_i(s_i^{k-1},\xi_{ij}^{k-1})),\\
\label{xk1}x_i^{k+1} &= \textmd{prox}_{r_i}^{\psi_i}\left(x_i^{k+\frac{1}{2}}\right),\\
\label{zk1} z_i^{k+1}&=z_i^k+\beta (x_i^{k+1}-z_i^k).
 \end{align}
\ENDFOR
\ENDFOR
\end{algorithmic}
\caption{Proposed SPPDM Algorithm}
\label{alg_ADMM2}
\end{algorithm}
Before showing how the algorithm is developed in Section \ref{subsec: alg development},
let us make a few comments about SPPDM.

 In Algorithm~\ref{alg_ADMM2}, $\alpha,\beta,\gamma,c,\kappa,\eta_k$ are some positive constant parameters that depend on the problem instance (such as the Lipschitz constants of $\{\nabla f_i\}$) and the graph Laplacian matrix).
Equations \eqref{mom}-\eqref{zk1} are the updates performed by each agent $i$ within the $k$th communication round, for $k=1,2,\ldots,$ and $i=1,\ldots,N$.
%Specifically,
%steps \eqref{mom}-\eqref{xk1} consists of the momentum, the consensus gradient update and the proximal gradient update, respectively.
Specifically, step \eqref{mom} is the introduced Nesterov momentum term $s_i^k$ for accelerating the algorithm convergence, where $\eta_k$ is the extrapolation coefficient at iteration $k$.
Step \eqref{xk12} shows how the neighboring variables $\{x_j\}_{j\in \mathcal{N}_i}$ are used for local gradient update. Note here that
in SPPDM the agent uses the sample average
$
  \frac{1}{|\mathcal{I}|} \sum_{j=1}^{|\mathcal{I}|} G_i(s_i^k,\xi_{ij}^k)
$
to approximate $\nabla f_i(s_i^k)$, where $\xi_{ij}^k\sim \mathcal{B}_i,~j=1,\ldots,|\mathcal{I}|,$ denotes the samples drawn by agent $i$ in the $k$th iteration. Besides, in \eqref{xk12}, both approximate gradients at $s_i^k$ and $s_i^{k-1}$ are used.
Step \eqref{xk1} performs the proximal gradient update with respect to the regularization term $r_i(x)$. In step \eqref{xk12}, the variable $\{z_i^k\}$ is a ``proximal" variable introduced for overcoming the non-convexity of $f_i$ (see \eqref{eqn: dec_proximal}), and is updated as in step \eqref{zk1}.

By stacking the variables for all $i=1,\ldots,N$, one can write \eqref{mom}-\eqref{zk1} in a vector form. Specifically, step \eqref{xk12} for $i=1,\ldots,N$, can be expressed compactly as
\begin{align}
\nonumber \vx^{k+\frac{1}{2}}=&~\vx^{k-1+\frac{1}{2}}+\Ub \vx^k-\tilde{\Ub}\vx^{k-1} \notag \\
   &+ \gamma \Psib^{-1}(\vs^{k} -\vs^{k-1} ) + \kappa \Psib^{-1}(\vz^{k} -\vz^{k-1} ) \notag \\
&- {\Psib}^{-1}( \bar G(\vs^k,\vxi^k)-  \bar G(\vs^{k-1},\vxi^{k-1})), \label{xk12_new0}
\end{align}
where $\Ub$ and $\tilde{\Ub}$ are two matrices satisfying %have the elements $w_{ij}$ and $\tilde{w}_{ij}$ as follows,
\begin{align}\label{U1}
&[\Ub]_{ij}=\left\{
\begin{array}{ll}
\frac{d_i}{\psi_i}(c-\alpha), & \hbox{$i=j$,} \\
\frac{c+\alpha}{\psi_i}, & \hbox{$i\neq j$ and $(i,j)\in\mathcal{E}$,}\\
0,& \hbox{\textmd{otherwise.}}
\end{array}
\right.  \\
&[\tilde{\Ub}]_{ij}=\left\{
\begin{array}{ll}
\frac{d_i c}{\psi_i}, & \hbox{$i=j$,} \\
\frac{c}{\psi_i}, & \hbox{$i\neq j$ and $(i,j)\in\mathcal{E}$,}\\
0& \hbox{\textmd{otherwise.}}
\end{array}\label{U2}
\right.
%\\
%&w_{ij}=\left\{
%\begin{array}{ll}
%\frac{1}{\psi_i}(\gamma+\kappa+d_i(c-\alpha)), & \hbox{$i=j$,} \\
%\frac{c+\alpha}{\psi_i}, & \hbox{$i\neq j$ and $(i,j)\in\mathcal{E}$,}\\
%0,& \hbox{\textmd{otherwise.}}
%\end{array}
%\right.
\end{align}
for all $i,j=1,\ldots,N$,
$\Psib$ is a diagonal matrix with its $i$th element being $\psi_i:=\gamma+2cd_i+\kappa$ for $i=1,\ldots,N$, and
\begin{align}\label{stochastic gradient}\bar G(\vs^k,\vxi^k):=\frac{1}{|\mathcal{I}|} \sum_{j=1}^{|\mathcal{I}|} G(\vs^k,\xib_{j}^k).
\end{align}

When the full gradients $\nabla f_i$ are available under the offline/batch setting, the approximate gradient $G_i$ in \eqref{xk12} and \eqref{xk12_new0} can be replaced by $\nabla f_i$. Then, the SPPDM algorithm reduces to the PPDM algorithm. %where the corresponding \eqref{xk12_new0}

%
%we assume that the exact gradients in \eqref{xk12} are known for offline setting, i.e., $G_i(s_i^k,\xi_{ij}^k)=\nabla f_i(s_i^k)$, then the SPPDM reduces to a deterministic  version, called proximal primal dual with momentum (PPDM), see Algorithm \ref{alg_ADMM3}.

%
%\begin{algorithm}
%\begin{algorithmic}
%\STATE Choose constants $\alpha,\beta,\gamma,c,\kappa,\eta_k$, let $\psi_i=\gamma+2cd_i+\kappa$, pick arbitrary initial $x_i^0$, set $s_i^0=x_i^0$, $i=1,\ldots,N$ and do
%\vspace{-0.4cm}
%\STATE\begin{align}
%\nonumber x_i^{\frac{1}{2}} &=(\gamma+cd_i+\kappa)\frac{x_i^0}{\psi_i}+\frac{c}{\psi_i}\sum_{j\in \mathcal{N}_i}x_j^0-\frac{1}{\psi_i}\nabla f_i(x_i^0),\\
%\nonumber x_i^1 &= \textmd{prox}_{r_i}^{\alpha_i}\left(x_i^{\frac{1}{2}}\right)
%\end{align}
%\FOR{$k=1,2,\ldots$}
%\FOR{$i=1,2,\ldots,N$ (in parallel)}
%\vspace{-0.4cm}
%\STATE \begin{align}
%\nonumber s_i^{k}&=x_i^k+\eta_k(x_i^k-x_i^{k-1}),\\
%\label{xk123} x_i^{k+\frac{1}{2}}&=x_i^{k-1+\frac{1}{2}}+\frac{d_i}{\psi_i}((c-\alpha)x_i^k-c x_i^{k-1}),\\
%\nonumber &~~~~+\frac{1}{\psi_i}\sum_{j\in \mathcal{N}_i}((c+\alpha)x_j^{k}-c x_j^{k-1})\\
%\nonumber &~~~~+\frac{1}{\psi_i}\big(\gamma(s_i^k-s_i^{k-1})+\kappa(z_i^k-z_i^{k-1})\big)\\
%\nonumber &~~~~-\frac{1}{\psi_i}(\nabla f_i(s_i^k)- \nabla f_i(s_i^{k-1}))),\\
%\label{xk1234} x_i^{k+1} &= \textmd{prox}_{r_i}^{\psi_i}\left(x_i^{k+\frac{1}{2}}\right),\\
%\nonumber z_i^{k+1}&=z_i^k+\beta (x_i^{k+1}-z_i^k).
% \end{align}
%\ENDFOR
%\ENDFOR
%\end{algorithmic}
%\caption{Proposed PPDM Algorithm}
%\label{alg_ADMM3}
%\end{algorithm}
%\subsection{Connection with PG-EXTRA}
\begin{remark}{\rm \label{rmk1}
	We show that the PPDM algorithm can have a close connection with the PG-EXTRA algorithm in \cite{shi2015proximal}.
	Specifically, let us set
	$\eta_k=0$ (no Nesterov momentum) and $\beta=1$ (no proximal variable). Then, we have $s_i^k=z_i^k=x_i^k$ for all $k,i$, and the momentum and proximal variable update in (\ref{mom}) and (\ref{zk1}) can be removed.
	As a result, \eqref{xk12_new0} reduces to
	\begin{align}
	\nonumber \vx^{k+\frac{1}{2}}=&~\vx^{k-1+\frac{1}{2}}+\Wb \vx^k-\tilde{\Wb}\vx^{k-1} \notag\\
	&- {\Psib}^{-1}( \nabla f(\vx^k)-  \nabla f(\vx^{k-1})), \label{xk12_new}
	\end{align}
	where $\Wb=\Ub+(\gamma+\kappa)\Psib^{-1}$ and $\tilde \Wb=\tilde \Ub+(\gamma+\kappa)\Psib^{-1}$.
	One can see that  \eqref{xk12_new} and \eqref{xk1} have an identical form as the PG-EXTRA algorithm in \cite[Eqn. (3a)-(3b)]{shi2015proximal}.
	Therefore, the proposed PPDM algorithm can be regarded as an accelerated version of the PG-EXTRA with extra capability to handle non-convex problems.
	One should note that, unlike \eqref{U1} and \eqref{U2}, the PG-EXTRA allows a more flexible choice of the mixing matrix $\Wb$, and thus it is also closely related to the GT based methods \cite{chang2020distributed}.
}
\end{remark}

\begin{remark}{\rm  \label{rmk2} The PPDM algorithm also has a close connection with the distributed Nesterov gradient  (D-NG) algorithm in \cite{jakovetic2014fast}.
Specifically, let us set $\alpha=c$ and $\beta=1$ (no proximal variable) and remove the non-smooth regularization term $r(\vx)$. Then, we have $z_i^k=x_i^k$ for all $k,i$, and the proximal gradient update \eqref{xk1} and  the proximal variable update (\ref{zk1}) can be removed.
Under the setting, as shown in  Appendix \ref{derivation of connection}, one can write  \eqref{xk12_new0} of the
PPDM algorithm as
\begin{align}
%\label{eqn: new1} \vs^{k} &= \vx^k+\eta_k(\vx^k-\vx^{k-1})\\
\label{eqn: new2} \vx^{k+1} &= \tilde{\mW} \vs^k-\Psib^{-1}\nabla f(\vs^k)+\mC^{k},
\end{align}
where
%$\tilde{\mW}=\Psib^{-1}(\gamma\mI+\kappa\mI)+\tilde \Ub$  and
%$\mC^{k}$ is a cumulative correction term with
$\mC^{k}=(\tilde \Ub(\vx^k-\vs^k)+\kappa(\vx^k-\vs^k)\Psib^{-1})-\sum_{t=0}^k(\mI-\tilde{\mW})\vx^t$ can regarded as a cumulative correction term.
Note that the D-NG algorithm in \cite[Eqn. (2)-(3)]{jakovetic2014fast} is
\begin{align}
\label{eqn: Dng1} \vs^{k} = \vx^k+\eta_k(\vx^k-\vx^{k-1}),\\
\label{eqn: Dng2} \vx^{k+1} = \tilde{\mW} \vs^k-\Psib^{-1} \nabla f(\vs^k).
\end{align}
One can see that  \eqref{eqn: Dng2} and \eqref{eqn: new2} have a similar form except for the correction term.
Note that the convergence of the D-NG algorithm is proved in \cite{jakovetic2014fast} only for convex problems with a diminishing step size. Therefore, the proposed PPDM algorithm is an enhanced counterpart of the D-NG algorithm with the ability to handle non-convex and non-smooth problems.
	}
\end{remark}

\subsection{Algorithm Development}\label{subsec: alg development}
%Let us introduce the incidence matrix $\tilde{\mA}\in \mathbb{R}^{|\mathcal{E}|\times n}$, which is a matrix with entires $\tilde{\mA}(k,i)=1$ and $\tilde{\mA}(k,j)=-1$ if $(i,j)\in \mathcal{E}$ with $j>i$, and all the rest of the entries being zero. Define the extended incidence matrix as $\mA:=\tilde{\mA}\otimes I_{n}$, where $\otimes$ denotes the Kronecker product. Thus, the other compact form about \eqref{equ_2}-\eqref{equ_3}  is
%\begin{subequations}\label{eqn: dec_compact}
%\begin{align}
%\label{equ_2_1a}\min_{\vx\in\mathcal{K}}& f(\vx)+r(\vx)\\
%\label{equ_3_2a} \textmd{s.t.} &~ \mA\vx=0.
%\end{align}
%\end{subequations}
In this subsection, let us elaborate how the SPPDM algorithm is devised.
Our proposed algorithm is inspired by the proximal AL framework in \cite{zhang2018proximal}. %, with additional consideration of the Nesterov momentum and stochastic gradients.
First, we introduce a proximal term $\vz$ to \eqref{eqn: dec_compact} as
\begin{subequations}\label{eqn: dec_proximal}
\begin{align}
\label{equ_2_11a}\min_{\vx,\vz}&~ f(\vx)+r(\vx)+\frac{\kappa}{2}\|\vx-\vz\|^2\\
\label{equ_3_22a} \textmd{s.t.} &~ \mA\vx=0,
\end{align}
\end{subequations}
where $\kappa>0$ is a  parameter. Obviously, \eqref{eqn: dec_proximal} is equivalent to  \eqref{eqn: dec_compact}. The purpose of adding the proximal term $\frac{\kappa}{2}\|\vx-\vz\|^2$ is to make the objective function in \eqref{equ_2_11a} strongly convex with respect to $\vx$ when $\kappa>0$ is large enough. Such strong convexity will be exploited for building the algorithm convergence.

Second, let us consider the AL function of \eqref{eqn: dec_proximal} as follows
\begin{align}
\nonumber L_c(\vx,\vz;\vlambda)=&f(\vx)+r(\vx)+\langle \vlambda, \mA \vx \rangle\\
\label{equ_2_L} &+ \frac{c}{2}\| \mA \vx\|^2+\frac{\kappa}{2}\|\vx-\vz\|^2,
\end{align}
where $\vlambda\in\mathbb{R}^{|\mathcal{E}|}$ is the Lagrangian dual variable, and $c>0$ is a positive penalty parameter.
Then, the Lagrange dual problem of \eqref{eqn: dec_proximal} can be expressed as
\begin{align}\label{maxmin}
  \max_{\vlambda}\min_{\vx, \vz}~ L_c(\vx,\vz;\vlambda).
\end{align}
We apply  the following inexact stochastic primal-dual updates with momentum for problem \eqref{maxmin}: for $k=0,1,2,\ldots$,
\begin{align}
\label{va}\vlambda^{k+1}&=\vlambda^k+\alpha \mA \vx^k,\\
\label{sa}\vs^k&=\vx^{k} +\eta_k(\vx^k-\vx^{k-1}),\\
\label{xa}\vx^{k+1}&=\mathop{\arg\min}_{\vx}g(\vx,\vx^k,\vs^k,\vz^k,\vxi^k;\vlambda^{k+1}),\\
\label{za}\vz^{k+1}&=\vz^k+\beta(\vx^{k+1}-\vz^k).
\end{align}
Specifically, \eqref{va} is the dual ascent step with $\alpha>0$ being the dual step size.
In \eqref{sa}, the momentum variable $\vs^k$ is introduced for the primal variable $\vx$.

To update $\vx$, we consider the inexact step as in \eqref{xa}
where $g(\vx,\vx^k,\vs^k,\vz^k,\vxi^k;\vlambda^{k+1})$ is a surrogate function given by
\begin{align}
\nonumber &g(\vx,\vx^k,\vs^k,\vz^k,\vxi^k;\vlambda^{k+1})\\
\nonumber=&\underbrace{f(\vs^k)+ \langle \bar G(\vs^k,\vxi^k),\vx-\vs^k\rangle +\frac{\gamma}{2}\|\vx-\vs^k\|^2}_{\rm (a)}\\
\nonumber&+ r(\vx)+\langle \vlambda^{k+1},\mA\vx\rangle\\
\label{fun_g}&+\underbrace{\frac{c}{2}\|\mA\vx\|^2+\frac{c}{2}\|\vx-\vx^k\|_{\mB^{\top}\mB}^2}_{\rm (b)}+\frac{\kappa}{2}\|\vx-\vz^k\|^2.
\end{align}
In \eqref{fun_g}, the term (a) is a quadratic approximation of $f$ at $\vs^k$ using the stochastic gradient $\bar G$, where $\gamma>0$ is a parameter. In term (b) of \eqref{fun_g}, $\mB$ is the signless incidence matrix of the graph $\Gc$, i.e., $\mB=|\mA|$, which satisfies
$\mA^{\top}\mA+\mB^{\top}\mB=2\mD$, where $\mD={\rm diag}\{d_1,\ldots,d_N\}$ is the degree matrix of $\Gc$.
As shown in \cite{hong2017prox}, the introduction of $\frac{c}{2}\|\vx-\vx^k\|_{\mB^{\top}\mB}^2$ can ``diagonalize"
$\frac{c}{2}\|\mA\vx\|^2$ and lead to distributed implementation of \eqref{xa}. In particular, one can show that
\eqref{xa} with \eqref{fun_g} can be expressed as
\begin{align}
\nonumber\vx^{k+1} &= \textmd{prox}_{r}^{\Psib}\Big(\Psib^{-1}\big(\gamma \vs^k+c\mB^{\top}\mB\vx^k+\kappa \vz^k
\\
\label{pro21}&\qquad\qquad- \bar G(\vs^k,\vxi^k)-\mA^{\top}\vlambda^{k+1}\big)\Big).
\end{align}
As seen, due to the graphical structure of $\mB^\top \mB$, each $x_i^{k+1}$ in \eqref{pro21} can be obtained in a distributed fashion using only $x_j^k$, $j\in \Nc_i$ from its neighbors.
Lastly, we update $\vz$ by applying the gradient descent to $L_c(\vx^{k+1},\vz;\vlambda^{k+1})$
with step size $\beta$, which then yields \eqref{za}.

% $\vx^{k+1}$, note that the gradient of the augmented Lagrangian \eqref{equ_2_L} at $\vz^k$ is $\kappa(\vz^k-
%\vx^{k+1})$. Thus, we update $\vz$ by gradient decent in \eqref{za}.

To show how \eqref{mom}-\eqref{zk1} are obtained, let $\vp^{k}=\mA^{\top} \vlambda^{k}$
and define
\begin{align}
  \vx^{k+\frac{1}{2}}=&\Psib^{-1}\big(\gamma \vs^k+c\mB^{\top}\mB\vx^k+\kappa \vz^k \notag \\
   &\qquad\qquad- \bar G(\vs^k,\vxi^k)-\vp^{k+1}\big).
\end{align}
Then, \eqref{va} can be replaced by
\begin{align}
\label{va2}\vp^{k+1}&=\vp^k+\alpha \mA^\top \mA \vx^k,
\end{align}
and \eqref{pro21} can be written as
\begin{align}\label{pro31}
\vx^{k+1} &= \textmd{prox}_{r}^{\Psib}\big(\vx^{k+\frac{1}{2}}\big).
\end{align}
Moreover, by subtracting  $\vx^{k-1+\frac{1}{2}}$ from $\vx^{k+\frac{1}{2}}$, one obtains
\begin{align}
\vx^{k+\frac{1}{2}}=&\vx^{k-1+\frac{1}{2}} + \gamma\Psib^{-1} (\vs^k - \vs^{k-1})+\kappa\Psib^{-1} (\vz^k-\vz^{k-1})\notag \\
& +c\Psib^{-1}\mB^{\top}\mB(\vx^k-\vx^{k-1})- \Psib^{-1}(\vp^{k+1} - \vp^{k})\notag \\
&- \Psib^{-1} (\bar G(\vs^k,\vxi^k) - \bar G(\vs^{k-1},\vxi^{k-1}) ). \label{pro41}
\end{align}
After substituting \eqref{va2} into \eqref{pro41}, we obtain
\begin{align}
\nonumber \vx^{k+\frac{1}{2}}=&~\vx^{k-1+\frac{1}{2}}+\Ub \vx^k-\tilde{\Ub}\vx^{k-1} \notag \\
&+ \gamma \Psib^{-1}(\vs^{k} -\vs^{k-1} ) + \kappa \Psib^{-1}(\vz^{k} -\vz^{k-1} ) \notag \\
&- {\Psib}^{-1}( \bar G(\vs^k,\vxi^k)-  \bar G(\vs^{k-1},\vxi^{k-1})), \label{xk12_new01}
\end{align}
which is exactly \eqref{xk12_new0} since $\Ub = c\Psib^{-1}\mB^{\top}\mB - \alpha \Psib^{-1}\mA^\top \mA$ and
$\tilde \Ub = c\Psib^{-1}\mB^{\top}\mB$
by \eqref{U1} and \eqref{U2}, respectively.

In summary, \eqref{va} and \eqref{xa} can be equivalently written as \eqref{xk12_new01} and \eqref{pro31}, and therefore we obtain
\eqref{sa}, \eqref{xk12_new01},  \eqref{pro31} and \eqref{za} as the algorithm updates, which correspond to \eqref{mom}-\eqref{zk1} in Algorithm \ref{alg_ADMM2}

Before ending the section, we remark that it is possible to
employ the existing stochastic primal-dual methods such as \cite{huang2018mini} for solving the non-smooth and non-convex problem \eqref{eqn: dec_compact}. However, these methods require strict conditions on $\Ab$. For example, the stochastic ADMM method in \cite{huang2018mini} requires $\Ab$ to have full rank, which cannot happen for the distributed optimization problem \eqref{eqn: dec_compact} since
the graph incidence matrix $\Ab$ for a connected graph must be rank deficient.

\section{Convergence Analysis}\label{sec: conv analysis}
In this section, we present the main theoretical results of the proposed SPPDM and PPDM algorithms
by establishing their convergence conditions and convergence rate.
\subsection{Assumptions}
We first make some proper assumptions on problem \eqref{eqn: dec_compact}. %\eqref{eqn: dec_compact}.

\begin{assumption}\label{assu}
%For any $i\in \{1,\ldots,N\}$,

\begin{enumerate}[(i)]
  \item The function $f(\vx)$ is a continuously differentiable function with Lipschitz continuous gradients, i.e., for constant $L>0$,
	\begin{align}\label{Lip}
	\|\nabla f(\vx)-\nabla f(\vy)\|\leq L \|\vx-\vy\|,
	\end{align}for all $\vx,\vy$.
	Moreover, assume that there exists a constant $\mu \geq -L$ (possibly negative) such that
	\begin{align}
	\label{lower} f(\vx)-f(\vy)-\langle \nabla f(\vy), \vx-\vy\rangle \geq \frac{\mu}{2}\|\vx-\vy\|^2,
	\end{align}	for all $\vx,\vy$.
	
  \item {The objective function $f(\vx)+r(\vx)$ is bounded from below in the feasible set $\{\vx|\mA\vx=0\}$, i.e.,
  $$
 f(\vx)+r(\vx)>\underline{f}>-\infty,
  $$
  for some constant $\underline{f}$.}
\end{enumerate}
\end{assumption}

\begin{assumption}  \label{assu2} The epigraph of each $r_i(x_i)$, i.e.,
	$\{(x_i,y_i)~|~ r_i(x_i)\leq y_i\}$,
	is a polyhedral set and has a compact form as %\emph{\cite{hong2017linear}}
	\begin{align}\label{compact}
	S_{x,i}x_i+S_{y,i}y_i\geq \zeta_i,
	\end{align}
	where $S_{x,i} \in \mathbb{R}^{q_i \times n}$, $S_{y,i}\in \mathbb{R}^{q_i}$ and $\zeta_i\in \mathbb{R}^{q_i}$ are some constant matrix and vectors.

%  \item $r_i(x_i)$ is a non-smooth convex component, the epigraph of each $r_i(x_i)$ is a polyhedral set with $\{(x_i,y_i)|x_i\in\mathcal{K}, r_i(x_i)\leq y_i\}$ and has a compact form as \emph{\cite{hong2017linear}}
%      \begin{align}\label{compact}
%      C_{x_i}x_i+C_{y_i}y_i\geq \zeta_i,
%      \end{align}
%      where $C_{x_i}$, $C_{y_i}$ and $\zeta_i$ are some constant matrices.

\end{assumption}
By \eqref{compact}, problem \eqref{eqn: dec_compact} can be written as
\begin{subequations}\label{eqn: dec_proximal1}
\begin{align}
\label{equ_3_1_1} \min_{\vx,\vy}&~ f(\vx)+\oneb^{\top}\vy\\
\label{equ_3_1_2} \textmd{s.t.} &~ \mA\vx=\zerob\\
\label{equ_3_1_3} &~\mS_{x}\vx+\mS_{y}\vy\geq \vzeta,
\end{align}
\end{subequations}
Here, $\vy=[y_1,\ldots,y_N]^{\top}$, $\mS_{x}={\rm diag}\{S_{x,1},\ldots,S_{x,N}\}$, $\mS_{y}={\rm diag}\{S_{y,1},\ldots,S_{y,N}\}$,
and $\vzeta=[\zeta_1^{\top},\ldots,\zeta_N^{\top}]^{\top}$.

Let $\vmu=[\mu_1,\ldots,\mu_q]^{\top}\in\mathbb{R}^{q}$, $q=\sum_{i=1}^N q_i$, be the dual variable associated with \eqref{equ_3_1_3}. Then, the Karush-Kuhn-Tucker (KKT) conditions of \eqref{eqn: dec_proximal1} are given by
\begin{subequations}\label{eqn: KKT}
\begin{align}
\label{gradient_x}&\nabla f(\vx)+\mA^{\top}\vlambda-\mS_{x}^{\top}\vmu=0,~\mS_{y}^{\top}\vmu=\oneb,\\
 \label{primal_feasiable}&\mA\vx=0,~\mS_{x}\vx+\mS_{y}\vy- \vzeta\geq0,~\vmu\geq 0,\\
 \label{CS}
&\mu_j[\mS_{x}\vx+\mS_{y}\vy- \vzeta]_j=0, ~j=1,\ldots,q.
\end{align}
\end{subequations}

For online/streaming learning, we also make the following standard assumptions that the gradient estimates are unbiased and have a bounded variance.
\begin{assumption}\label{assu3}
The stochastic gradient estimate $G_i(x,\xi)$ satisfies
\begin{align}
\label{unbiased}&\mathbb{E}[G_i(x,\xi)]=\nabla f_i(x)\\
\label{variance}&\mathbb{E}[\|G_i(x,\xi)-\nabla f_i(x)\|^2]\leq \sigma^2,
\end{align}
for all $x$, where $\sigma>0$ is a constant, and the expectation $\E$ is with respect to the random sample $\xi\sim \mathcal{B}_i$.
\end{assumption}
It is easy to check that the gradient estimate of the mini-batch samples satisfies
\begin{align}
   \mathbb{E}\bigg[\bigg \| \frac{1}{|\mathcal{I}|} \sum_{j=1}^{|\mathcal{I}|}  G_i(x,\xi_j)-\nabla f_i(x)\bigg \|^2\bigg ] \leq \sigma^2/|\mathcal{I}|.
\end{align}

\subsection{Convergence Analysis of SPPDM}
We define the following term
\begin{align}\label{QQ}
\!\!\!Q(\vx,\vlambda)=&\|\vx-\textmd{prox}_{r}^1(\vx-\nabla f(\vx)-\mA^{\top}\vlambda)\|^2
+\|\mA\vx\|^2
\end{align}
as the optimally gap for a primal-dual solution $(\vx,\vlambda)$ of problem \eqref{eqn: dec_compact}.
Obviously, one can shown that when $Q(\vx^\star,\vlambda^\star)=0$, $(\vx^\star,\vlambda^\star)$ is a KKT solution of \eqref{eqn: dec_compact} which satisfies the conditions in \eqref{eqn: KKT} together with some $\vy^\star$ and $\vmu^\star$. We define that $(\vx^\star,\vlambda^\star)$ is an $\epsilon$-stationary solution of \eqref{eqn: dec_compact} if $Q(\vx^\star,\vlambda^\star)<\epsilon$.

The convergence  result is stated in the following theorem.
\begin{theorem}\label{The_conver}
Assume that Assumptions \ref{assu}-\ref{assu3} hold true, and let parameters satisfy
\begin{align}
  &\kappa>-\mu,~\gamma>3L,\label{parameter condition} \\
  &\label{etak}\eta_k\leq\sqrt{\frac{\kappa+2c+\gamma-3L}{2(\gamma-\mu+3L)}} :=\bar \eta ,
\end{align}
moreover, let $0<\alpha\leq c$ and $\beta>0$ be both sufficiently small (see \eqref{alpha_q}  and \eqref{equ_beta}). Then, for a sequence $\{\vx^k,\vz^k,\vlambda^k\}$ generated by Algorithm \ref{alg_ADMM2}, it holds that
\begin{align}
\label{QQQ1}
 \!\!\! \!\! \min_{k=0,\ldots,K-1} \E[ Q(\vx^k,\vlambda^{k+1})] \leq C_0\left(\frac{\phi^0-\underline{f}}{K}+\frac{C_1N\sigma^2}{|\mathcal{I}|}\right), %\leq \epsilon.
\end{align}
where $C_0$ and $C_1$ are some positive constants depending on the problem parameters (see  \eqref{C0} and \eqref{C1}). In addition, $\phi^0$ is a constant defined in  \eqref{potential}.
\end{theorem}

To prove Theorem \ref{The_conver}, the key is to define a novel stochastic potential function $\mathbb{E}[\phi^{k+1}]$ in \eqref{potential} and analyze the conditions for which $\mathbb{E}[\phi^{k+1}]$ descends monotonically with the iteration number $k$ (Lemma \ref{Lemma_potential}). To achieve the goal, several approximation error bounds for the primal variable $\vx^k$ (Lemma \ref{primal_error_bound}) and the dual variable $\vlambda^{k}$  (Lemma \ref{dual_error_bound}) are derived. Interested readers may refer to Appendix \ref{appendix thm1} for the details.

By Theorem \ref{The_conver}, we immediately obtain the following corollary.
\begin{corollary}\label{conv to stationary sol}Let
\begin{align}
\label{bach_K}|\mathcal{I}|\geq\frac{2NC_0C_1\sigma^2}{\epsilon}~{\rm and~} K\geq \frac{2C_0(\phi^0-\underline{f})}{\epsilon}.
\end{align}
Then,
\begin{align}
\label{QQQ}\min_{k=0,\ldots,K-1} \E[ Q(\vx^k,\vlambda^{k+1})]  \leq \epsilon,
\end{align}
that is, an $\epsilon$-stationary solution of problem \eqref{eqn: dec_compact} can be obtained in an expected sense.
\end{corollary}

\begin{remark}{\rm
Given a mini-batch size $|\mathcal{I}|=\Omega(1/\epsilon)$, Corollary \ref{conv to stationary sol} implies that the proposed
SPPDM algorithm has the convergence rate of $\mathcal{O}(1/\epsilon)$ to obtain an $\epsilon$-stationary solution.
As a result, the corresponding communication complexity of the
SPPDM algorithm is $\mathcal{O}(|\Ec|/\epsilon)$  while the
computational complexity is $\mathcal{O}(N|\mathcal{I}|/\epsilon)
=
\mathcal{O}(N/\epsilon^2)$.
%for obtaining an $\epsilon$-stationary point, respectively.
As shown in Table \ref{Tab:01}, the communication complexity $\mathcal{O}(1/\epsilon)$ of the SPPDM algorithm is smaller than $\mathcal{O}(1/\epsilon^2)$ of D-PSGD \cite{lian2017can},  D$^2$ \cite{tang2018d}, GNSD \cite{lu2019gnsd} and  R-SGD-M \cite{yu2019linear}. The STOC-ADMM \cite{huang2018mini} has the same computation and communication complexity orders as the SPPDM algorithm, but it is not applicable to  \eqref{eqn: dec_compact}.
}\end{remark}

\subsection{Convergence Analysis of PPDM}

When the full gradient $\nabla f(\vx^k)$ is available for the PPDM algorithm, one can deduce a similar convergence result.
%
%
%in offline setting, in order to show  the convergence of Algorithm \ref{alg_ADMM3}, we consider another potential function,
%\begin{align}
%\nonumber\phi^{k+1}&\triangleq L_c(\vx^{k+1},\vz^{k+1};\vlambda^{k+1})-2 d(\vz^{k+1};\vlambda^{k+1})\\
% \label{potential3}&\qquad\qquad+2P(\vz^{k+1})+\tau\|\vx^{k+1}-\vx^{k}\|^2.
%\end{align}
%The lower bound of $\phi^{k}$ is $\phi^{*}$. Moreover, the new optimality gap is defined as
%\begin{align}
%\nonumber Q(\vx^k,\vlambda^{k+1})=&\min_{t\in\{1,\ldots,k\}}\|\mA\vx^{\top}\|^2+\|\vx^{\top}-\\
% \label{QQ3}&\textmd{prox}_{r}^1(\vx^{\top}-\nabla_{\vx}f(\vx^{\top})-\mA^{\top}\vlambda^{t+1})\|^2.
%\end{align}
%Similar to the analysis for online setting, we have the following theorem for offline setting.
\begin{theorem}\label{thm PPDM}
Assume Assumptions \ref{assu}-\ref{assu2} and the same conditions in \eqref{parameter condition}, \eqref{etak},  \eqref{equ_beta} and \eqref{alpha_q} hold true.
 \begin{itemize}
   \item Every limit point of the sequence $\{\vx^k,\vz^k,\vlambda^k\}$ generated by the PPDM algorithm is a KKT solution
of \eqref{eqn: dec_compact}. %, i.e., $\lim_{k\into }Q(\vx^\star,\vlambda^\star)=0$

   \item Given $K\geq \frac{C_0(\phi^0-\underline{f})}{\epsilon}$, we have
\begin{align}
\nonumber \min_{k=0,\ldots,K-1} Q(\vx^k,\vlambda^{k+1}) \leq C_0\left(\frac{\phi^0-\underline{f}}{K}\right)\leq \epsilon.
\end{align}
%Thus, the communication complexity is $\mathcal{O}(1/\epsilon)$.
 \end{itemize}
\end{theorem}
\noindent The proof is presented in Appendix \ref{appendix thm2}..

To our knowledge, Theorem \ref{The_conver} and Theorem \ref{thm PPDM} are the first results that show the $\mathcal{O}(1/\epsilon)$ communication complexity of the distributed primal-dual method with momentum for non-convex and non-smooth problems.
%Although it possess the same speedup property with primal dual method in \cite{zhang2018proximal}, the simulations show the proposed method can enjoy faster convergence in practice.
Numerical results in the next section will demonstrate that the SPPDM and PPDM algorithms can exhibit favorable convergence behavior than the existing methods.

\section{Numerical Results}\label{sec: simulation}
In this section, we examine the numerical performance of the proposed SPPDM/PPDM algorithm
and present comparison results with the existing methods.

\subsection{Distributed Non-Convex Truncated Losses}
We consider a linear regression model $y_j={ h_j}^{\top}x_{*}+\nu_j$, $j=1,\ldots,M$, where $M$ is the number of data samples.
Here $y_j$ is the observed data sample and $h_j\in \mathbb{R}^n$ is the input data; $x_{*}\in \mathbb{R}^n$ is the ground truth; $\nu_j$ is the additive random noise.

Let $H:=[h_1,\ldots,h_M]^{\top}=[H_1^{\top},\ldots,H_N^{\top}]^{\top}\in \mathbb{R}^{M\times n}$, where each $H_i\in \mathbb{R}^{m\times n}$ corresponds to the data matrix owned by agent $i$ which has $m=M/N$ data points. The entries of $H$ are generated independently following the standard Gaussian distribution.
The ground truth $x_{*}$ is a $S$-sparse vector to be estimated, whose non-zero entries are generated from the uniform distribution $U[-1,1]$.
The noise $\nu_j$ follows the Gaussian distribution $\mathcal{N}(0,4)$. Then the data samples $y_j,~j=1,\ldots,M$, are generated by the above linear model.

Consider the following distributed regression problem with a nonconvex truncated loss \cite{xu2018learning}
\begin{align}
\label{truncated_loss}\min_{x\in[-1,1]}\sum_{i=1}^N\big(f_i(x)+\varsigma_i \|x\|_1\big),
\end{align}
where
\begin{align*}
f_i(x) = \frac{\rho}{2N_i}\sum_{j=1}^{N_i}\log\left(1+\frac{\|y_j-a_j^{\top}x\|^2}{\rho}\right),
\end{align*}
and $\rho$ is a parameter to determine the truncation level. We set $m=150$, $n=256$, $S=16$, and $\rho=3$. Moreover, we consider a circle graph with $N=20$ agents.

%To test the performance of the proposed algorithm,

For the online setting, we compare the SPPDM algorithm (Algorithm \ref{alg_ADMM2}) with PSGD \cite{bianchi2012convergence} and STOC-ADMM \cite{huang2018mini}.
For the offline setting, we compare the PPDM algorithm with Prox-DGD \cite{zeng2018nonconvex}, PG-EXTRA \cite{shi2015proximal}, Prox-ADMM \cite{zhang2018proximal} and STOC-ADMM \cite{huang2018mini}.
Note that theoretically PG-EXTRA and STOC-ADMM are not guaranteed to converge for the non-convex problem \eqref{eqn: dec_compact}. We implement these two methods simply for comparison purpose.

For the PG-EXTRA, we choose the stepsize $\ell=0.05$ according to the sufficient condition
suggested in  \cite{shi2015proximal}.
According to their convergence conditions, the diminishing step size $\ell=\frac{1}{3\sqrt{k+100}}$  is used for the PSGD and Prox-DGD.
The primal and dual stepsize for the STOC-ADMM is chosen according to the convergence condition
suggested in \cite{huang2018mini}.   %The parameters of Prox-ADMM are same to those of the proposed PPDM.

For PSGD,  Prox-DGD and PG-EXTRA, the mixing matrix follows the metropolis weight
\begin{equation}\label{eqn: W}
[\mW]_{i j} \triangleq\left\{\begin{array}{ll}\frac{1}{\max \left\{d_{i}, d_{j}\right\}+1}, & \text { for }(i, j) \in \mathcal{E}, \\ 0, & \text { for }(i, j) \notin \mathcal{E} \text { and } i \neq j. \\ 1-\sum_{j \neq i} w_{i j}, & \text { for } i=j\end{array}\right.
\end{equation}

%Moreover, PG-EXTRA also uses the above mixing matrix.
%Since theoretically PG-EXTRA \emph{only} works for the convex problem, it is not guaranteed to converge when applied to \eqref{truncated_loss}. Here, we implement it for comparison purpose.

If not specified, the parameters of the SPPDM/PPDM and  the Prox-ADMM are given as $\alpha=2$, $\kappa=1$, $c=2$, $\gamma=3$, $\beta=0.9$\footnote{
 By  analysis, the Hessian matrix for the function $f_i(x)$ is $\frac{1}{N_i}\sum_{j=1}^{N_i}\frac{\rho h_jh_j^T(\rho-\|h_j^Tx-y_j\|^2)}{(\rho+\|h_j^Tx-y_j\|^2)^2}$. It shows that the maximum  eigenvalue of this Hessian matrix is smaller than 1 ($L<1$) with the given parameter. Thus, the parameters of SPPDM/PPDM
satisfy the conditions stated in Theorem
\ref{The_conver}.}
For the proposed SPPDM, we consider two cases about $\eta_k$, one is $\eta_k=0$ without momentum, and the other is  based on the Nesterov's extrapolation technique, i.e.,
\begin{align*}
\textstyle \eta_k=\frac{\theta_{k-1}-1}{\theta_k},\quad \theta_{k+1}=\frac{1+\sqrt{1+4\theta_k^2}}{2},
\end{align*} with $\theta_{-1}=\theta_0=1$.
When $\eta_k=0$, we denote SPPDM as SPPD.
%The other parameters in SPPD are same to those in SPPDM.

Define $\bar{x}=\frac{1}{N}\sum_{i=1}^Nx_i$. The stationarity error and consensus error are defined
below
\begin{align*}
\textmd{stationarity error} &= \|\bar{x}-\textmd{prox}_r[\bar{x}-\nabla f(\bar{x})]\|^2,\\
\textmd{consensus error}  &= \frac{1}{N}\sum_{i=1}^N\|x_i-\bar{x}\|^2.
\end{align*}
We run 10 independents trials for each algorithm with randomly
generated data and random initial values. The convergence curves
obtained by averaging over all 10 trials are plotted in Figs. \ref{fig:stationarity}-\ref{fig:consensus_full}.

In Fig. \ref{fig:stationarity}, we observe that the SPPDM, SPPM and STOC-ADMM all perform better than the PSGD in terms of stationarity error and consensus error. The reason is that these methods all use constant step sizes rather than the diminishing step size.  In addition, the proposed SPPDM has better performance than SPPD and STOC-ADMM, due to the Nesterov momentum.

The impact of the mini-batch size $|\mathcal{I}|$, and parameters  $\gamma$ and $c$ are analyzed in Fig.
\ref{fig:stationarity_parameter}. One can see that the larger mini-batch size
we use, the smaller error we can achieve, which corroborates Corollary \ref{conv to stationary sol}. With the same  mini-batch size, the larger values of $\gamma$ and $c$ correspond to smaller primal and dual step sizes. Thus, the SPPDM with larger values of $\gamma$ and $c$ has slower convergence; whereas as seen from the figures,
larger values of $\gamma$ and $c$ can lead to smaller stationarity error and consensus error.

For the offline setting, the comparison results of the proposed PPDM with the existing methods are shown in Fig. \ref{fig:consensus_full}.
It can be observed that the proposed PPDM enjoys the fastest convergence. Compared with the Prox-ADMM, it is clear to see the advantage of the PPDM with momentum for speeding up the algorithm convergence.
\begin{figure}[t]
\begin{center}
\includegraphics[width=1\linewidth]{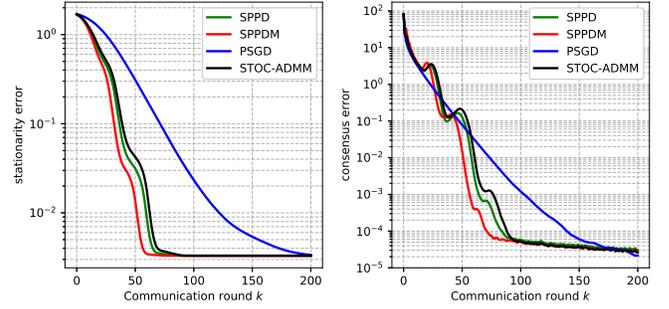}
\vspace{-0.8cm}
\caption{Comparison of proposed SPPDM with the PSGD and STOC-ADMM in terms of stationarity and consensus error; the batch size is $|\mathcal{I}|=100$. }
\label{fig:stationarity}
\end{center}
\end{figure}
\begin{figure}[t]
\begin{center}
\includegraphics[width=1\linewidth]{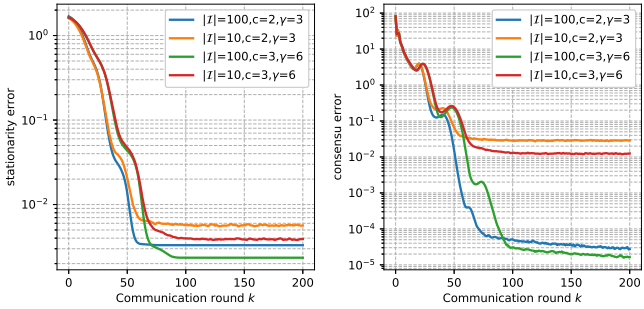}
\vspace{-0.8cm}
\caption{Comparison of the proposed SPPDM with different parameters in terms of stationarity and consensus error.}
\label{fig:stationarity_parameter}
\end{center}\vspace{-0.6cm}
\end{figure}
\begin{figure}[t]
\begin{center}
\includegraphics[width=1\linewidth]{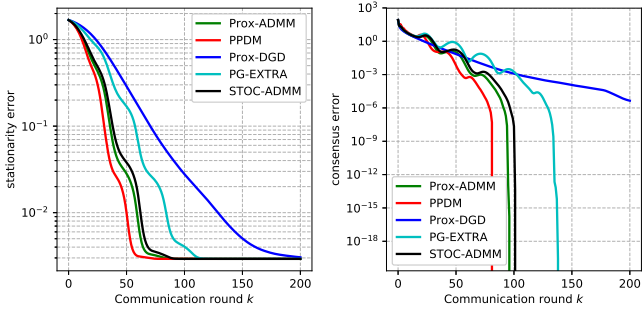}
\vspace{-0.8cm}
\caption{Comparison of proposed PPDM with the existing methods in terms of stationarity and consensus error.}
\label{fig:consensus_full}
\end{center}
\end{figure}
\subsection{Distributed Neural Network}
In this simulation, our  task is to classify handwritten digits from the MNIST dataset. The local loss function
$f_i(\theta_i)$ in each node is the cross-entropy function. In this example, we do not consider nonsmooth term and inequality constraint set. Thus, many existing methods, D-PSGD \cite{lian2017can}, D$^{2}$ \cite{tang2018d} and
PR-SGD-M \cite{yu2019linear}  can be applied to train a classification DNN.

Assume the neural network contains one hidden layer with 500 neurons. The $6\times 10^4$ training samples are divided into 10 subsets and assigned to the $N=10$ agents in two ways.
The first is the \emph{IID} case, where the samples are sufficiently shuffled, and then partitioned into 10 subsets with equal size ($m=6000$).
The second is the \emph{Non-IID} case, where we first sort the samples according to their labels, divide it into 20 shards of size 3000, and assign each of 10 agents 2 shards. Thus most agents have samples of two digits only.

The communication graph is also a circle.
We compare the SPPDM with the D-PSGD \cite{lian2017can}, D$^{2}$ \cite{tang2018d} and PR-SGD-M \cite{yu2019linear}.
The same mixing matrix in \eqref{eqn: W} is used for the three methods. Moreover, a fixed step size of  $\ell=0.05$ is used to ensure the convergence of these three methods in the simulation.
For the proposed SPPDM, we set parameter $c=1$, $\gamma=3$, $\alpha=0.001$, $\kappa=0.1$, $\beta=0.9$, and
$\eta_k=0.8$. %When $\eta_k=0$, we call SPPDM as SPPD.
The batch size is $|\mathcal{I}|=128$.

We calculate the loss value and the classification accuracy based on the average model $\bar{\theta}=1/N\sum_{i=1}^N\theta_i$.  Fig. \ref{fig:loss} and Fig. \ref{fig:acc} show the training loss and the classification accuracy for the  IID case and Non-IID case by averaging over all 5 trials, respectively.
From Fig. \ref{fig:loss},
we see that D$^2$ and D-PSGD have a similar performance; meanwhile, the proposed SPPD performs better than these two methods.
Besides, one can see that SPPDM and PR-SGD-M enjoy fast decreasing of the loss function and increasing of the classification accuracy, respectively. The reason is that both SPPDM and PR-SGD-M use the momentum technique. We should point out that the communication overhead of PR-SGD-M
is twice of the SPPDM since the PR-SGD-M requires
the agents to exchange not only the variable $x_i$ but also the momentum variables. Lastly,
comparing the SPPDM  with  SPPD, it shows again that the momentum techniques can accelerate the algorithm convergence.

From Fig. \ref{fig:acc} for the non-IID case, we can observe that the
D$^2$ performs better than D-PSGD and SPPD. In fact, by comparing
the curves in Fig. \ref{fig:loss}  with those in Fig.  \ref{fig:acc},
one can see that the
convergence curve of D$^2$ remains almost the same due to the use of the variance reduction technique
whereas D-PSGD and SPPD deteriorate under the non-IID setting.
As seen, the convergence of SPPDM is also slowed, but it still performs best among the methods under test.

\begin{figure}[t]
\begin{center}
\includegraphics[width=1\linewidth]{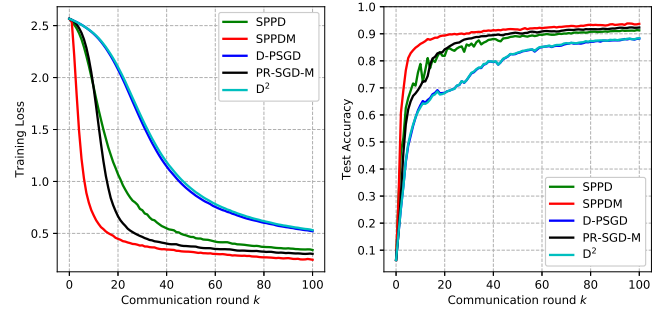}
\vspace{-0.8cm}
\caption{Comparison of proposed { SPPDM/SPPD} algorithms with different methods under the IID case.}
\label{fig:loss}
\end{center}\vspace{-0.7cm}
\end{figure}
\begin{figure}[t]
\begin{center}
\includegraphics[width=1\linewidth]{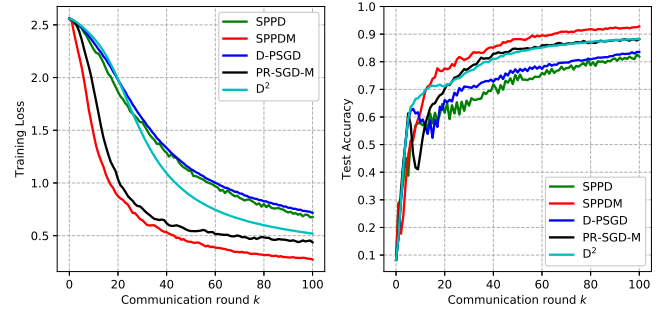}
\vspace{-0.8cm}
\caption{Comparison of proposed { SPPDM/SPPD} algorithms with different methods under the Non-IID case. }
\label{fig:acc}
\end{center}\vspace{-0.6cm}
\end{figure}

\section{Conclusion}\label{sec:conclusion}
In this paper, we have proposed a distributed stochastic proximal primal-dual algorithm with momentum for minimizing a non-convex and non-smooth function \eqref{eqn: dec_compact} over a connected multi-agent network.
We have shown (in Remark \ref{rmk1} and Remark \ref{rmk2}) that the proposed algorithm has a close connection with some of the existing algorithms that are for convex and smooth problems, and therefore can be regarded as an enhanced counterpart of these existing algorithms. Theoretically,  under Assumptions \ref{assu}-\ref{assu3}, we have built the convergence conditions of the proposed algorithms in Theorem \ref{The_conver} and Theorem \ref{thm PPDM}. In particular, we have shown that the proposed SPPDM can achieve an $\epsilon$-stationary solution  with $\mathcal{O}(1/\epsilon^2)$ computational complexity and $\mathcal{O}(1/\epsilon)$ communication complexity, where the latter is better than many of the existing methods which have  $\mathcal{O}(1/\epsilon^2)$ communication complexity (see Table 1). Experimental results have demonstrated that the proposed algorithms with momentum can effectively speed up the convergence. For distributed learning under non-IID data distribution (Fig. \ref{fig:acc}), we have also shown the proposed SPPDM performs better than the existing methods.

As future research directions, one may further relax Assumption \ref{assu3} to accommodate a larger class of regularization functions. Besides, it will be also interesting to investigate the  trade-off between the communication complexity (as measured by the number of bits exchanged) and the convergence. Moreover, we will analytically investigate how data distribution affect the algorithm convergence and improve the robustness of the algorithms against to unbalanced and non-IID data distribution in the future.

%
%
% {\purple
% Compared with the traditional decentralized parallel stochastic gradient descent method,
%in this paper, we have developed a distributed stochastic and deterministic consensus optimization method with momentum for a non-convex and non-smooth problem, but the traditional decentralized parallel stochastic gradient descent method only can deal with the smooth nonconvex problem under decreasing step size. Under constant step size,
%we have proved that the proposed SPPDM can achieve an $\epsilon$-stationary solution  with $\mathcal{O}(1/\epsilon^2)$. In addition, PPDM can obtain an $\epsilon$-stationary solution  with $\mathcal{O}(1/\epsilon)$.
%Numerical results are presented to illustrate the proposed method performs better. The future work may be combing the proposed method with the federated learning.}

		\vspace{-0.0cm}
		\appendices
%		{\setcounter{equation}{0}
%			\renewcommand{\theequation}{A.\arabic{equation}}
%\section{Appendix}

\section{Derivation of \eqref{eqn: new2}}\label{derivation of connection}
When we remove the non-smooth regularization term $r(\vx)$, the proximal gradient
update can be removed. Assume $\beta=1$, then \eqref{pro21} can be written as
\begin{align}
\nonumber\vx^{k+1} &= \Psib^{-1}\big(\gamma \vs^k+c\mB^{\top}\mB\vx^k+\kappa \vx^k
- \nabla f(\vs^k)-\mA^{\top}\vlambda^{k+1}\big)\\
\nonumber &=\Psib^{-1}\big( \tilde{\mW} \vs^k-\Psib^{-1}\nabla f(\vs^k)+\mC^{k}\big),
\end{align}
where
\begin{align}
\label{wb}\tilde \Wb&=c\Psib^{-1}\mB^{\top}\mB+(\gamma+\kappa)\Psib^{-1},\\
\label{Ck}\mC^{k} & = \Psib^{-1}(c\mB^{\top}\mB(\vx^k-\vs^k)+\kappa(\vx^k-\vs^k))-\Psib^{-1}\mA^{\top}\vlambda^{k+1}.
\end{align}
Using the definition of $\tilde \Ub=c\Psib^{-1}\mB^{\top}\mB$ in \eqref{U2},  we rewrite \eqref{wb} as
\begin{align*}
\tilde \Wb&=\tilde \Ub+\Psib^{-1}(\gamma+\kappa)\mI.
 \end{align*}
 According the definition of $\Psib=(\gamma+\kappa)\mI+c(\mA^{\top}\mA+\mB^T\mB)$ and $c=\alpha$, we have
 \begin{align}
\label{IW} \mI-\tilde \Wb=\Psib^{-1}\Psib-\tilde \Wb=\alpha \Psib^{-1}\mA^{\top}\mA.
 \end{align}
Based on
\begin{align}
\mA^{\top}\vlambda^{k+1}=\vp^{k+1}=\vp^{k}+\alpha\mA^{\top}\mA\vx^k,
\end{align}
if $\vp^{0}=0$, and applying \eqref{IW}, we obtain
\begin{align*}
\Psib^{-1}\mA^{\top}\vlambda^{k+1}=\sum_{t=0}^k\alpha\Psib^{-1}\mA^{\top}\mA\vx^t=\sum_{t=0}^k(\mI-\tilde \Wb)\vx^t.
\end{align*}
By substituting the above equality into \eqref{Ck}, we obtain \eqref{eqn: new2}.

\section{Proof of Theorem \ref{The_conver}}\label{appendix thm1}
Let us recapitulate the augmented Lagrange function in \eqref{equ_2_L} below
\begin{align}
\nonumber L_c(\vx,\vz;\vlambda)=&f(\vx)+r(\vx)+\langle \vlambda, \mA \vx \rangle\\
\label{equ_2_L2} &+ \frac{c}{2}\| \mA \vx\|^2+\frac{\kappa}{2}\|\vx-\vz\|^2.
\end{align}
We introduce some auxiliary functions as follows
\begin{align}
\label{dual}&d(\vz;\vlambda)=\min_{\vx}L_c(\vx,\vz;\vlambda)\\
\label{dual_opt}&\vx(\vz;\vlambda)=\mathop{\arg\min}_{\vx}L_c(\vx,\vz;\vlambda)\\
\label{proximal}&P(\vz)=\min_{\substack{\mA\vx=0}}f(\vx)+r(\vx)+\frac{\kappa}{2}\|\vx-\vz\|^2\\
\label{proximal_opt}&\vx(\vz)=\mathop{\arg\min}_{\substack{\mA\vx=0}}f(\vx)+r(\vx)+\frac{\kappa}
{2}\|\vx-\vz\|^2.
\end{align}
%By weak duality, we have
%\begin{align}
%\label{weak_dual}L_c(\vx,\vz;\vlambda)\geq d(\vz;\vlambda), P(\vz)\geq d(\vz;\vlambda).
%\end{align}
Besides, we define the full gradient iterate  $\hat{\vx}^{k+1}$ and $\hat{\vz}^{k+1}$,
\begin{align}
%\label{hatl}\hat{\vlambda}^{k+1}&:=\hat{\vlambda}^k+\alpha \mA \hat{\vx}^k\\
\label{hatx} \hat{\vx}^{k+1}&:=\mathop{\arg\min}_{\vx}g(\vx,\vw^k;\vlambda^{k+1})\\
\label{hatz} \hat{\vz}^{k+1}&:=\vz^k+\beta(\hat{\vx}^{k+1}-\vz^k),
\end{align}
where $\vw^k=[\vx^k,\vs^k,\vz^k]$ and
\begin{align}
\nonumber &g(\vx,\vw^k;\vlambda^{k+1})\\
\nonumber&=f(\vs^k)+\langle \nabla f(\vs^k),\vx-\vs^k\rangle+\frac{\gamma}{2}\|\vx-\vs^k\|^2+ r(\vx)\\
\label{fun_g1}&+\langle \vlambda^{k+1},\mA\vx\rangle+\frac{c}{2}\|\mA\vx\|^2+\frac{c}{2}\|\vx-\vx^k\|_{\mB^{\top}\mB}^2+\frac{\kappa}{2}\|\vx-\vz^k\|^2.
\end{align}
We also define
\begin{align*}
g(\vx,\vw^k,\vxi^k;\vlambda^{k+1})&:=g(\vx,\vx^k,\vs^k,\vz^k,\vxi^k;\vlambda^{k+1})
\end{align*} for \eqref{fun_g} at our disposal.

\subsection{Some Error Bounds}
Firstly,
we  show the upper bound between $\vx^{k+1}$ and $\hat{\vx}^{k+1}$.
\begin{lemma}\label{lemma_xx} Suppose Assumption \ref{assu3} holds, we have
	\begin{align}
	\label{xhatx}&\mathbb{E}[\|\vx^{k+1}-\hat{\vx}^{k+1}\|^2]\leq \frac{{ N\sigma^2}}{(\gamma+2c+\kappa)^2|\mathcal{I}|}.
	\end{align}
\end{lemma}
\begin{proof}
	According to \eqref{unbiased}-\eqref{variance} in Assumption \ref{assu3}, we know
	\begin{align}
	\label{uv}\mathbb{E}\Big[\big\|\frac{1}{|\mathcal{I}|} \sum_{j=1}^{|\mathcal{I}|} G(\vs^k,\vxi_j^k)-\nabla f(\vx)\big\|^2\Big]\leq \frac{N}{|\mathcal{I}|}\sigma^2.
	\end{align}
	In addition, like \eqref{pro21}, the proximity form of \eqref{hatx} is
	\begin{align}
	\nonumber\hat{\vx}^{k+1} &= \textmd{prox}_{r}^{\Psib}\Big(\Psib^{-1}\big(\gamma \vs^k+c\mB^{\top}\mB\vx^k+\kappa \vz^k
	\\
	\label{pro211}&\qquad\qquad-\nabla f(\vs^k)-\mA^{\top}\vlambda^{k+1}\big)\Big).
	\end{align}
	Using \eqref{pro21}, \eqref{uv}-\eqref{pro211}  and applying the nonexpansive property of the proximal operator (see for example \cite[p. 340]{Rockafellar1970}) we then obtain \eqref{xhatx}.
\end{proof}

\begin{lemma}\label{primal_error_bound}
Suppose $\kappa>-\mu$. There exists some positive constants $\sigma_1$, $\sigma_2$  such that the following primal error bound holds
\begin{align}
\label{primal_error}
\!\!\!\!\|\vx^k-\vx(\vz^k;\vlambda^{k+1})\|\leq \sigma_1\|\vx^k-\hat{\vx}^{k+1}\|+\sigma_2\|\vx^k-\vs^k\|.
\end{align}
\end{lemma}
\begin{proof}
Based on $\kappa>-\mu$, we know that $L_c$ in \eqref{equ_2_L2} is strongly convex in $\vx$ with modulus $\kappa+\mu$ and Lipschitz constant $\kappa+L+c\sigma_A^2$, where $\sigma_A$ is the spectral norm of the
matrix. Thus, we can apply  \cite[Theorem 3.1]{pang1987posteriori} to upper bound the distance between $\vx^k$ and the optimal solution $\vx(\vz^k;\vlambda^{k+1})$
\begin{align}
\label{primal_bound} &\|\vx^k-\vx(\vz^k;\vlambda^{k+1})\|\leq \varrho\|\tilde{\nabla}_{\vx}L_c(\vx^k,\vz^k;\vlambda^{k+1})\|,
\end{align}
where $\varrho=\frac{\kappa+L+c\sigma_A^2+1}{\kappa+\mu}$ and
\begin{align}
\nonumber &\tilde{\nabla}_{\vx}L_c(\vx,\vz;\vlambda)=\vx-\textmd{prox}_{r}^{\Psib}(\vx-\nabla_{\vx}(L_c(\vx,\vz;\vlambda)-r(\vx)))
\end{align} is known as the proximal gradient.

We can bound $\|\tilde{\nabla}_{\vx}L_c(\vx^k,\vz^k;\vlambda^{k+1})\|$ as follows
\begin{align}
\nonumber&\|\tilde{\nabla}_{\vx}L_c(\vx^k,\vz^k;\vlambda^{k+1})\|\\
\nonumber&=\|\vx^k-\textmd{prox}_{r}^{\Psib}(\vx^k-\nabla_{\vx}(L_c(\vx^k,\vz^k;\vlambda^{k+1})-r(\vx^k)))\|\\
\nonumber&\leq\|\vx^k-\hat{\vx}^{k+1}\|\\
\nonumber&~~~+\|\hat{\vx}^{k+1}-\textmd{prox}_{r}^{\Psib}(\vx^k-\nabla_{\vx}(L_c(\vx^k,\vz^k;\vlambda^{k+1})-r(\vx^k)))\|\\
\nonumber&=\|\vx^k-\hat{\vx}^{k+1}\|\\
\nonumber&~~~+\Big\|\textmd{prox}_{r}^{\Psib}(\hat{\vx}^{k+1}-\nabla_{\vx}(g(\hat{\vx}^{k+1},\vw^k;\vlambda^{k+1})-r(\hat{\vx}^{k+1})))\\
\nonumber&~~~-\textmd{prox}_{r}^{\Psib}(\vx^k-\nabla_{\vx}(L_c(\vx^k,\vz^k;\vlambda^{k+1})-r(\vx^k)))\Big\|\\
\nonumber&\leq (2+2cd_{\max}+\gamma+\kappa)\|\vx^k-\hat{\vx}^{k+1}\|+(\gamma+L)\|\vx^k-\vs^k\|,
\end{align}
where $d_{\max}=\max\{d_1,\ldots,d_N\}$; the second equality is obtained by using the optimality condition of $\hat{\vx}^{k+1}$ in \eqref{hatx}, and the second inequality is  based  on the nonexpansive property of the proximal operator. Denote
\begin{align}
\label{eqn: sigma1}\sigma_1&=\varrho(2+2cd_{\max}+\gamma+\kappa),\\
\label{eqn: sigma2}\sigma_2&=\gamma+L.
\end{align}
The proof is complete.
\end{proof}
\begin{lemma} \label{lemma_error_bound}
(Lemma 3.2 in \cite{zhang2018proximal})
Suppose $\kappa>-\mu$, and Assumption \ref{assu} holds. There exists some positive constants $\sigma_3,\sigma_4$ such that the following error bounds hold
\begin{align}
\label{v_lemma}\|\vlambda_1-\vlambda_2\|&\geq\sigma_3\|\vx(\vz;\vlambda_1)-\vx(\vz;\vlambda_2)\|\\
\label{z_lemma}\|\vz_1-\vz_2\|&\geq \sigma_4\|\vx(\vz_1;\vlambda)-\vx(\vz_2;\vlambda)\|,
\end{align}
where
\begin{align}
\label{eqn: sigma3}
\sigma_3&=(\kappa+\mu)/\sigma_A\\
\label{eqn: sigma4}\sigma_4&=(\kappa+\mu)/\kappa.
\end{align}
\end{lemma}
\begin{lemma}\label{dual_error_bound}
%(See Lemma 3.2 in \cite{zhang2020})
Suppose that  Assumptions \ref{assu}-\ref{assu2} hold and $\kappa>\mu$. Then, there exist some positive scalars  $\sigma_5$, $\Delta$ such that the following dual error bound holds
	\begin{align}
	\label{dual_lemma}\|\vx(\vz,\vlambda)-\vx(\vz)\|\leq \sigma_5\|\mA\vx(\vz;\vlambda)\|, ~for~ any ~\vz, ~\vlambda.
	\end{align}
where $\sigma_5$ depends only on the constants $L,\kappa,\sigma_A,\mu$ and the matrices $\mA, \mS_x, \mS_y$.
%	whenever
%	\begin{align}
%	\label{cond_dual_lemma}\|\mA\vx(\vz;\vlambda)\|\leq\Delta,~dist(\vz,\mX^{*})\leq\Delta,
%	\end{align}
%	where  $\Lambda^{*}(\vz)$  denotes the solution set of dual multipliers \eqref{proximal}.
\end{lemma}
\begin{proof}
{ The lemma is an extension of \cite[Lemma 3.2]{zhang2020}, where the non-smooth term $r(\vx)$ of \eqref{eqn: dec_compact} is limited to an indicator function of a polyhedral set. Due to limited space, the proof details are relegated to the supplementary document \cite{Wang2020Supplementary}.
}	
%{\red Note that \cite[Lemma 3.2]{zhang2020} only applies to the problem \eqref{eqn: dec_compact} that the non-smooth term $r(\vx)$ is an indicator function of a polyhedral set. However, we consider a more general problem that
%non-smooth term $r(\vx)$ (for example $\ell_1$-norm) has a polyhedral epigraph. Thus, we need to extend \cite[Lemma 3.2]{zhang2020} to Lemma \ref{dual_error_bound}.}
\end{proof}

\subsection{Decent Lemmas}
In order to show  the convergence of Algorithm \ref{alg_ADMM2}, we consider a new potential function,
\begin{align}
\nonumber\mathbb{E}[\phi^{k+1}]&\triangleq\mathbb{E}[L_c(\vx^{k+1},\vz^{k+1};\vlambda^{k+1}) +\tau\|\vx^{k+1}-\vx^{k}\|^2 ]\\
\label{potential}&~~~+\mathbb{E}[2P(\vz^{k+1}) -2 d(\vz^{k+1};\vlambda^{k+1})],
\end{align}
for some $\tau >0$.
By the weak duality, we have
\begin{align}
\label{weak_dual}L_c(\vx,\vz;\vlambda)\geq d(\vz;\vlambda), P(\vz)\geq d(\vz;\vlambda).
\end{align}
Thus, we have $\mathbb{E}[\phi^{k}]\geq \mathbb{E}[P(\vz^{k})]$. According to the definition of  $P(\vz^k)$ in \eqref{proximal} and Assumption \ref{assu} (ii), we obtain $P(\vz^k)\geq \underline{f}$. As a result, $\mathbb{E}[\phi^{k}]$  is bounded below by $\underline{f}$.
\begin{lemma}\label{lemma1}
For a sequence $\{\vx^k,\vz^k,\vlambda^k\}$ generated by Algorithm \ref{alg_ADMM2}, if $\kappa>-\mu$, $\gamma>3L$, $0<\beta<1$ and
\begin{align}
\label{etak2}0 \leq \eta_k\leq\sqrt{\frac{\kappa+2c+\gamma-3L}{2(\gamma-\mu+3L)}} :=\bar \eta,
\end{align}
there exit some positive constants $\tau$, $\hat{\sigma}_1$ and $\hat{\sigma}_2$ such that
\begin{align}
\label{eqn: hat_sigma1}\hat{\sigma}_1&\triangleq\frac{\kappa+2c+\gamma-3L}{2}-2\tau\geq0\\
\label{eqn: hat_sigma2}\hat{\sigma}_2& \triangleq\frac{\mu-\gamma-3L}{2}{\bar \eta}^2+\tau\geq0,
\end{align}
then
\begin{align}
\nonumber &\mathbb{E}[L_c(\vx^k,\vz^k;\vlambda^k)+\tau\|\vx^{k}-\vx^{k-1}\|^2]\\
\nonumber &~~~ -\mathbb{E}[L_c(\vx^{k+1},\vz^{k+1};\vlambda^{k+1})-\tau\|\vx^{k+1}-\vx^{k}\|^2]\\
\nonumber&\geq-\alpha \mathbb{E}[\|\mA \vx^k\|^2]+\frac{\kappa}{2\beta}\mathbb{E}[\|\vz^k-\vz^{k+1}\|^2]\\
\nonumber&~~~+ \hat{\sigma}_1\mathbb{E}[\|\vx^k-\hat{\vx}^{k+1}\|^2]+\hat{\sigma}_2\mathbb{E}[\|\vx^k-\vx^{k-1}\|^2]\\
\label{Lall} & ~~~ -\left(\frac{1}{2\gamma}+\frac{3L+4\tau}{2(\gamma+2c+\kappa)^2}\right)\frac{N\sigma^2}{|\mathcal{I}|}.
\end{align}
\end{lemma}
\begin{proof}
Firstly, according to \eqref{equ_2_L} and \eqref{va}, we have
\begin{align}
\label{Lv}
\!\!\!\!\mathbb{E}[L_c(\vx^k,\vz^k;\vlambda^k)-L_c(\vx^{k},\vz^k;\vlambda^{k+1})]= -\alpha \mathbb{E}[\|\mA \vx^k\|^2].
\end{align}
Secondly, we have
\begin{align}
\nonumber&\mathbb{E}[L_c(\vx^{k},\vz^k;\vlambda^{k+1})-L_c(\vx^{k+1},\vz^k;\vlambda^{k+1})]\\
\nonumber&= \mathbb{E}[L_c(\vx^{k},\vz^k;\vlambda^{k+1})-g(\vx^k,\vw^k;\vlambda^{k+1})]\\
\nonumber&~~~+\mathbb{E}[g(\vx^k,\vw^k\vlambda^{k+1})-g(\hat{\vx}^{k+1},\vw^k;\vlambda^{k+1})]\\
\nonumber&~~~+\mathbb{E}[g(\hat{\vx}^{k+1},\vw^k;\vlambda^{k+1})-g(\vx^{k+1},\vw^k,\vxi^k;\vlambda^{k+1})]\\
\label{Lx} &~~~+\mathbb{E}[g(\vx^{k+1},\vw^k,\vxi^k;\vlambda^{k+1})-L_c(\vx^{k+1},\vz^k;\vlambda^{k+1})].
\end{align}
Next, we bound each of the terms in the right hand side of \eqref{Lx}.
Based on the definition of function $g$ in \eqref{fun_g1}, we have
\begin{align}
\nonumber &\mathbb{E}[L_c(\vx^{k},\vz^k;\vlambda^{k+1})-g(\vx^k,\vw^k;\vlambda^{k+1})]\\
\nonumber &= \mathbb{E}[f(\vx^k)-f(\vs^k)-\langle \nabla f(\vs^k), \vx^k-\vs^k\rangle-\frac{\gamma}{2}\|\vx^k-\vs^k\|^2]\\
\label{Lcg}&\geq \frac{\mu-\gamma}{2}\mathbb{E}[\|\vx^k-\vs^k\|^2],
\end{align}
where the inequality comes from \eqref{lower} in Assumption \ref{assu}.
Using the strongly convexity of the objective function $g$ (with modulus $\kappa+2c+\gamma$) and the definition of $\hat{\vx}^{k+1}$ in \eqref{hatx}, we can obtain
\begin{align}
\nonumber & \mathbb{E}[g(\vx^k,\vw^k;\vlambda^{k+1})-g(\hat{\vx}^{k+1},\vw^k;\vlambda^{k+1})]\\
\label{Lx2} &\geq \frac{\kappa+2c+\gamma}{2}\mathbb{E}[\|\vx^k-\hat{\vx}^{k+1}\|^2].
\end{align}
In addition, we have
\begin{align}
\nonumber &\mathbb{E}[g(\hat{\vx}^{k+1},\vw^k;\vlambda^{k+1})-g(\vx^{k+1},\vw^k,\vxi^k;\vlambda^{k+1})]\\
\nonumber&=\mathbb{E}[g(\hat{\vx}^{k+1},\vw^k,\vxi^k;\vlambda^{k+1})-g(\vx^{k+1},\vw^k,\vxi^k;\vlambda^{k+1})]\\
\label{Lx22}&\geq 0,
\end{align}
where the first equality dues to  \eqref{unbiased} in Assumption \ref{assu3}, and the above inequality dues to $\vx^{k+1}$ is the optimal solution in \eqref{xa}. Lastly, we can bound
\begin{align}
\nonumber &\mathbb{E}[g(\vx^{k+1},\vw^k,\vxi^k;\vlambda^{k+1})-L_c(\vx^{k+1},\vz^k;\vlambda^{k+1})]\\
\nonumber &=\mathbb{E}[f(\vs^k)]+\frac{1}{|\mathcal{I}|} \sum_{j=1}^{|\mathcal{I}|}\mathbb{E}[\langle G(\vs^k,\vxi_j^k),\vx^{k+1}-\vs^k\rangle]\\
\nonumber &~~~+\mathbb{E}\left[\frac{\gamma}{2}\|\vx^{k+1}-\vs^k\|^2]+\frac{c}{2}\|\vx^{k+1}-\vx^k\|_{\mB^{\top}\mB}^2-f(\vx^{k+1})\right]\\
\nonumber &\geq \frac{1}{|\mathcal{I}|} \sum_{j=1}^{|\mathcal{I}|}\mathbb{E}[\langle G(\vs^k,\vxi_j^k)-\nabla f(\vs^k),\vx^{k+1}-\vs^k\rangle ]\\
\nonumber &~~~ +\frac{\gamma-L}{2}\mathbb{E}[\|\vx^{k+1}-\vs^k\|^2]+\frac{c}{2}\mathbb{E}[\|\vx^{k+1}-\vx^k\|_{\mB^{\top}\mB}^2]\\
 \label{Lx3}&\geq  -\frac{N\sigma^2}{2\gamma|\mathcal{I}|} -\frac{L}{2}\mathbb{E}[\|\vx^{k+1}-\vs^k\|^2],
\end{align}
where the first inequality is obtained by applying  the descent lemma \cite[Lemma1.2.3]{Nesterov2004}
\begin{align*}
 f(\vx^{k+1})\leq  f(\vs^k) + \langle \nabla f(\vs^k),\vx^{k+1}-\vs^k\rangle +\frac{L}{2}\|\vx^{k+1}-\vs^k\|^2
\end{align*} owing to gradient Lipschitz continuity in \eqref{Lip};  the second inequality holds by using the Young's inequality $a^{\top}b\geq -\frac{\|a\|^2}{2\gamma}-\frac{\gamma\|b\|^2}{2}$ and \eqref{variance} in Assumption \ref{assu3}. Using the convexity of the operator $\|\cdot\|^2$, we have
\begin{align}
\nonumber \|\vx^{k+1}-\vs^k\|^2\leq &3\|\vx^{k+1}-\hat{\vx}^{k+1}\|^2+3\|\hat{\vx}^{k+1}-\vx^{k}\|^2\\
 \label{abc}&+3\|\vx^{k}-\vs^k\|^2.
\end{align}
 Substituting \eqref{xhatx} and \eqref{abc} into \eqref{Lx3} gives rise to
\begin{align}
\nonumber &\mathbb{E}[g(\vx^{k+1},\vw^k,\vxi^k;\vlambda^{k+1})-L_c(\vx^{k+1},\vz^k;\vlambda^{k+1})]\\
\nonumber   &\geq  -\left(\frac{1}{2\gamma}+\frac{3L}{2(\gamma+2c+\kappa)^2}\right)\frac{N\sigma^2}{|\mathcal{I}|}\\
\label{Lx33}&~~~ -\frac{3L}{2}\mathbb{E}[\|\hat{\vx}^{k+1}-\vx^k\|^2]-\frac{3L}{2}\mathbb{E}[\|\vs^{k}-\vx^k\|^2].
\end{align}
By further substituting \eqref{Lcg}-\eqref{Lx22} and \eqref{Lx33} into \eqref{Lx},  we obtain
\begin{align}
\nonumber&\mathbb{E}[L_c(\vx^{k},\vz^k;\vlambda^{k+1})+\tau\|\vx^k-\vx^{k-1}\|^2\\
\nonumber&~~~-L_c(\vx^{k+1},\vz^k;\vlambda^{k+1})-\tau\|\vx^{k+1}-\vx^{k}\|^2]\\
\nonumber& \geq\frac{\mu-\gamma}{2}\mathbb{E}[\|\vx^k-\vs^k\|^2]+\frac{\kappa+2c+\gamma}{2}\mathbb{E}[\|\vx^k-\hat{\vx}^{k+1}\|^2]\\
\nonumber& ~~~-\frac{3L}{2}\mathbb{E}[\|\hat{\vx}^{k+1}-\vx^k\|^2]-\frac{3L}{2}\mathbb{E}[\|\vs^{k}-\vx^k\|^2]\\
\nonumber&~~~-\left(\frac{1}{2\gamma}+\frac{3L}{2(\gamma+2c+\kappa)^2}\right)\frac{\sigma^2}{|\mathcal{I}|}
+\tau\mathbb{E}[\|\vx^k-\vx^{k-1}\|^2]\\
\nonumber&~~~-2\tau\mathbb{E}[\|\vx^{k+1}-\hat{\vx}^{k+1}\|^2]-2\tau\mathbb{E}[\|\hat{\vx}^{k+1}-\vx^{k}\|^2]\\
\nonumber&=  \hat{\sigma}_1\mathbb{E}[\|\vx^k-\hat{\vx}^{k+1}\|^2]+\hat{\sigma}_2\mathbb{E}[\|\vx^k-\vx^{k-1}\|^2]\\
\label{Lx_new} &~~~ -\left(\frac{1}{2\gamma}+\frac{3L+4\tau}{2(\gamma+2c+\kappa)^2}\right)\frac{N\sigma^2}{|\mathcal{I}|},
\end{align}
where $\hat{\sigma}_1$ and $\hat{\sigma}_2$ are defined in \eqref{eqn: hat_sigma1} and \eqref{eqn: hat_sigma2}, respectively, and the equality is obtained by applying \eqref{sa}.

Thirdly, according to the definition of the $\vz$ update in \eqref{za}, we have
\begin{align}
\nonumber &\mathbb{E}[L_c(\vx^{k+1},\vz^k;\vlambda^{k+1})-L_c(\vx^{k+1},\vz^{k+1};\vlambda^{k+1})]\\
\label{Lz}&\geq
\frac{\kappa}{2\beta}(2-\beta)\mathbb{E}[\|\vz^k-\vz^{k+1}\|^2]\notag \\
&\geq \frac{\kappa}{2\beta}\mathbb{E}[\|\vz^k-\vz^{k+1}\|^2],
\end{align}
where the last inequality is due to $0<\beta<1$.
By combining \eqref{Lv}, \eqref{Lx_new} and \eqref{Lz}, we obtain  \eqref{Lall}. Besides, \eqref{eqn: hat_sigma1} and \eqref{eqn: hat_sigma2} implies \eqref{etak2}.
\end{proof}
%\subsection{Proof of Lemma \ref{Lemma_potential}.}

%In addition, let us define
%\begin{align}
%\label{equ_M}M=\max_{\vx_1,\vx_2\in\mathcal{K}}\|\vx_1-\vx_2\|,\zeta=\min\Big\{\Delta,\delta\big(\frac{\Delta\sqrt{1-\beta}}{7}\big)\Big\},
%\end{align}
%where $\Delta$ is defined in Lemma \ref{dual_error_bound} and $\delta$ is defined in Lemma \ref{lemma6}.
%Next, let us show the following important lemma.
\begin{lemma}\label{Lemma_potential}
	Under Assumptions \ref{assu}-\ref{assu3}, if $\kappa>-\mu$, $\gamma>3L$, $\eta_k$ is a constant satisfies the condition \eqref{etak}, and
	\begin{align}
    \label{alpha_q} 0<\alpha\leq \min\left\{\frac{\hat{\sigma}_1}{4 \sigma_A\sigma_1^2},\frac{\hat{\sigma}_2}{4\sigma_A^2\sigma_2^2\eta_k^2},c\right\},\\
	\label{equ_beta}0<\beta<\min\left\{\frac{\alpha}{12\kappa\sigma_5^2},\frac{\sigma_4}{36},1\right\},
	\end{align}
	where $\sigma_1$, $\sigma_2$, $\sigma_4$ and $\sigma_5$ are constants denoted in \eqref{eqn: sigma1} and \eqref{eqn: sigma2}, \eqref{eqn: sigma4} and \eqref{dual_lemma}, respectively. Then we have
\begin{align}
\nonumber&\mathbb{E}[\phi^k-\phi^{k+1}]\\
\nonumber&\geq\frac{\kappa(1-\beta)\beta}{4}\mathbb{E}[\|\hat{\vx}^{k+1}-\vz^k\|^2]+\frac{\alpha}{2}\mathbb{E}[
\|\mA\vx(\vz^k,\vlambda^{k+1})\|^2]\\
\label{phiphi4}&+\frac{\hat{\sigma}_1}{2}\mathbb{E}[\|\vx^k-\hat{\vx}^{k+1}\|^2]+\frac{\hat{\sigma}_2}{2}
\mathbb{E}[\|\vx^k-\vx^{k-1}\|^2]-\frac{C_1N\sigma^2}{|\mathcal{I}|},
\end{align}
where $\hat{\vx}^{k+1}$ and $\hat{\vz}^{k+1}$ are defined in \eqref{hatx} and \eqref{hatz}, and
	\begin{align}
	\label{C1} C_1=\left(\frac{1}{2\gamma}+\frac{6L+8\tau+\kappa(1-\beta)}{4(\gamma+2c+\kappa)^2}\right).
	\end{align}
\end{lemma}
\begin{proof}
From the definition of $d(\vz;\vlambda)$ in \eqref{dual},
we have
\begin{align}
\nonumber&\mathbb{E}[d(\vz^{k};\vlambda^{k+1})-d(\vz^{k};\vlambda^{k})]\\
\nonumber &=\mathbb{E}[L_c(\vx(\vz^{k};\vlambda^{k+1}),\vz^{k};\vlambda^{k+1})
-L_c(\vx(\vz^{k};\vlambda^{k}),\vz^{k};\vlambda^{k})]\\
\nonumber &\geq \mathbb{E}[L_c(\vx(\vz^{k};\vlambda^{k+1}),\vz^{k};\vlambda^{k+1})
-L_c(\vx(\vz^{k};\vlambda^{k+1}),\vz^{k};\vlambda^{k})]\\
\nonumber&=\alpha \mathbb{E}[\langle \mA\vx^k,\mA\vx(\vz^k,\vlambda^{k+1})\rangle],
\end{align}
where the inequality is due to $\vx(\vz^{k};\vlambda^{k})=\arg\min_{\vx}L_c(\vx,\vz^{k};\vlambda^{k})$  and the second equality comes from the iterates in \eqref{va}. Using a similar technique, we have
\begin{align}
\nonumber&\mathbb{E}[d(\vz^{k+1};\vlambda^{k+1})-d(\vz^{k};\vlambda^{k+1})]\\
\nonumber &\geq\frac{\kappa}{2}\mathbb{E}[(\vz^{k+1}-\vz^k)^{\top}(\vz^{k+1}+\vz^k-2\vx(\vz^{k+1},\vlambda^{k+1}))].
\end{align}
Combing the above two inequalities, we know
\begin{align}
\nonumber&\mathbb{E}[d(\vz^{k+1};\vlambda^{k+1})-d(\vz^{k};\vlambda^{k})]\\
\label{dual_mi}&\geq \alpha \mathbb{E}[\langle \mA\vx^k,\mA\vx(\vz^k,\vlambda^{k+1})\rangle]\\
\nonumber&~~~+\frac{\kappa}{2}\mathbb{E}[(\vz^{k+1}-\vz^k)^{\top}(\vz^{k+1}+\vz^k-2\vx(\vz^{k+1},\vlambda^{k+1}))].
\end{align}

Based on  Danskin's theorem \cite[Proposition B.25]{Bertsekas99} in convex analysis and $P(\vz)$ defined in \eqref{proximal} with $\kappa > -\mu$, we can have
$$
\nabla P(\vz^k)=\kappa(\vz^k-\vx(\vz^k)).
$$
Thus, it shows
\begin{align}
\nonumber&\|\nabla P(\vz^k)- \nabla P(\vz^{k+1})\|\\
\nonumber&\leq\kappa\|\vz^k-\vz^{k+1}\|+\kappa\|\vx(\vz^{k+1})-\vx(\vz^k)\|\\
\nonumber& \leq\kappa\tilde{\sigma}_4\|\vz^{k+1}-\vz^k\|,
\end{align}
where $\tilde{\sigma}_4=1+\sigma_4^{-1}$ and the final inequality is  due to Lemma \ref{lemma_error_bound}. The above inequality  shows the gradient of $P(\vz^k)$ is Lipschitz continuous, which therefore it satisfies the descent lemma
\begin{align}
 \label{proximal_mi}&\mathbb{E}[P(\vz^{k+1})-P(\vz^{k})]\\
\nonumber&\leq \mathbb{E}[\kappa(\vz^{k+1}-\vz^k)^{\top}(\vz^k-\vx(\vz^k))]
+\frac{\kappa\tilde{\sigma}_4}{2}\mathbb{E}[\|\vz^{k+1}-\vz^k\|^2].
\end{align}

By combining \eqref{dual_mi}, \eqref{proximal_mi} and \eqref{Lall}, we obtain
\begin{align}
\nonumber&\mathbb{E}[\phi^k-\phi^{k+1}]\\
\nonumber&\geq-\alpha \mathbb{E}[\|\mA \vx^k\|^2]+\frac{\kappa}{2\beta}\mathbb{E}[\|\vz^k-\vz^{k+1}\|^2]\\
\nonumber&~~~+ \hat{\sigma}_1\mathbb{E}[\|\vx^k-\hat{\vx}^{k+1}\|^2]-2\mathbb{E}[\kappa(\vz^{k+1}-\vz^k)^{\top}(\vz^k-\vx(\vz^k))]\\
\nonumber &~~~ +\hat{\sigma}_2\mathbb{E}[\|\vx^k-\vx^{k-1}\|^2]-\left(\frac{1}{2\gamma}+\frac{3L+4\tau}{2(\gamma+2c+\kappa)^2}\right)\frac{N\sigma^2}{|\mathcal{I}|}\\
\nonumber&~~~-\kappa\tilde{\sigma}_4\mathbb{E}[\|\vz^{k+1}-\vz^k\|^2]+2\alpha \mathbb{E}[\langle \mA\vx^k,\mA\vx(\vz^k,\vlambda^{k+1})\rangle]\\
\nonumber&~~~+\kappa\mathbb{E}[(\vz^{k+1}-\vz^k)^{\top}(\vz^{k+1}+\vz^k-2\vx(\vz^{k+1},\vlambda^{k+1}))]\\
\nonumber&=\alpha\mathbb{E}[\|\mA\vx(\vz^k,\vlambda^{k+1})\|^2]
-\alpha\mathbb{E}[\|\mA(\vx^k-\vx(\vz^k,\vlambda^{k+1}))\|^2] \\
\nonumber&~~~+\hat{\sigma}_1\mathbb{E}[\|\vx^k-\hat{\vx}^{k+1}\|^2] +(\frac{\kappa}{2\beta}+\kappa-\kappa\tilde{\sigma}_4)\mathbb{E}[\|\vz^{k+1}-\vz^k\|^2]\\
\nonumber&~~~ +2\kappa\mathbb{E}[(\vz^{k+1}-\vz^k)^{\top}(\vx(\vz^k)-\vx(\vz^{k+1};\vlambda^{k+1}))]\\
\label{phiphi}&~~~+\hat{\sigma}_2\mathbb{E}[\|\vx^k-\vx^{k-1}\|^2]-\left(\frac{1}{2\gamma}+\frac{3L+4\tau}{2(\gamma+2c+\kappa)^2}\right)\frac{N\sigma^2}{|\mathcal{I}|},
\end{align}
 where  the equality comes from  completing the square
\begin{align}
 &\mathbb{E}[\|\mA(\vx^k-\vx(\vz^k,\vlambda^{k+1}))\|^2] \notag \\
 &=  \mathbb{E}[\|\mA \vx^k\|^2-2\langle\mA \vx^k,\mA vx(\vz^k,\vlambda^{k+1})+\|\mA\vx(\vz^k,\vlambda^{k+1})\|^2\rangle].\notag
\end{align}

We  further bound the right-hand-side terms of \eqref{phiphi}.
By using the Young's inequality,  we have
\begin{align}
\nonumber&2(\vz^{k+1}-\vz^k)^{\top}(\vx(\vz^k)-\vx(\vz^k;\vlambda^{k+1}))\\
\nonumber&\geq -\frac{\|\vz^{k+1}-\vz^k\|^2}{6\beta}-6\beta\|\vx(\vz^k)-\vx(\vz^k;\vlambda^{k+1})\|^2\\
\label{Yongs}&\geq -\frac{\|\vz^{k+1}-\vz^k\|^2}{6\beta}-6\beta\sigma_5^2\|\mA\vx(\vz;\vlambda)\|^2,
\end{align}
where the lase inequality dues to \eqref{dual_lemma} in Lemma \ref{dual_error_bound}.
Besides, using the error bound \eqref{z_lemma} in Lemma \ref{lemma_error_bound}, we have
\begin{align}
\nonumber&(\vz^{k+1}-\vz^k)^{\top}(\vx(\vz^k;\vlambda^{k+1})-\vx(\vz^{k+1};\vlambda^{k+1}))\\
\nonumber&\geq -\|\vz^{k+1}-\vz^k\|\|\vx(\vz^k;\vlambda^{k+1})-\vx(\vz^{k+1};\vlambda^{k+1})\|\\
\label{vzz}&\geq -\frac{1}{\sigma_4}\|\vz^{k+1}-\vz^k\|^2.
\end{align}
Also, based on the error bound \eqref{primal_error} in Lemma \ref{primal_error_bound}, we obtain
\begin{align}
\nonumber &\|\mA(\vx^k-\vx(\vz^k,\vlambda^{k+1}))\|^2\\
\label{vxx}&\leq 2\sigma_A^2\sigma_1^2\|\vx^k-\hat{\vx}^{k+1}\|^2+2\sigma_A^2\sigma_2^2\|\vx^k-\vs^k\|^2.
\end{align}
%where $\sigma_A$ is the spectral norm of the matrix $\mA$.

By substituting \eqref{Yongs}, \eqref{vzz} and \eqref{vxx} into \eqref{phiphi}, we therefore obtain
\begin{align}
\nonumber&\mathbb{E}[\phi^k-\phi^{k+1}]\\
\nonumber&\geq (\alpha-6\kappa\beta\sigma_5^2) \mathbb{E}[\|\mA\vx(\vz^k,\vlambda^{k+1})\|^2]\\
\nonumber&~~~+(\frac{\kappa}{2\beta}+\kappa-\kappa\tilde{\sigma}_5-\frac{\kappa}{6\beta}-\frac{2\kappa}{\sigma_4})\mathbb{E}[\|\vz^{k+1}-\vz^k\|^2]\\
\nonumber&~~~+\left(\hat{\sigma}_1-2\alpha\sigma_A^2\sigma_1^2\right)\mathbb{E}[\|\vx^k-\hat{\vx}^{k+1}\|^2]\\
\nonumber&~~~+(\hat{\sigma}_2-2\alpha\sigma_A^2\sigma_2^2\eta_k^2)\mathbb{E}[\|\vx^k-\vx^{k-1}\|^2]\\
\label{phiphi1}&~~~ -\left(\frac{1}{2\gamma}+\frac{3L+4\tau}{2(\gamma+2c+\kappa)^2}\right)\frac{N\sigma^2}{|\mathcal{I}|}.
\end{align}
From \eqref{equ_beta}, we know $\beta<\frac{\sigma_4}{36}$. By recalling $\tilde{\sigma}_4=1+\sigma_4^{-1}$, we have
$$
\frac{\kappa}{2\beta}+\kappa-\kappa\tilde{\sigma}_4-\frac{\kappa}{6\beta}-\frac{2\kappa}{\sigma_4}\geq \frac{\kappa}{4\beta}.
$$
As $\beta<\frac{\alpha}{12\kappa\sigma_5^2}$ by \eqref{equ_beta}, we have
$$
\alpha-6\kappa\beta\sigma_5^2\geq \frac{\alpha}{2}.
$$
Similarly, based on \eqref{alpha_q}, we have
\begin{align*}
\hat{\sigma}_1-2\alpha\sigma_A^2\sigma_1^2\geq \frac{\hat{\sigma}_1}{2},~
 \hat{\sigma}_2-2\alpha\sigma_A^2\sigma_2^2\eta_k^2\geq \frac{\hat{\sigma}_2}{2}.
\end{align*}
Thus, it follows from \eqref{phiphi1} that
\begin{align}
\nonumber&\mathbb{E}[\phi^k-\phi^{k+1}]\\
\nonumber&\geq\frac{\kappa}{4\beta}\mathbb{E}[\|\vz^{k+1}-\vz^k\|^2]+\frac{\alpha}{2}\mathbb{E}[\|\mA\vx(\vz^k,\vlambda^{k+1})\|^2]\\
\nonumber&~~~+\frac{\hat \sigma_1}{2}\mathbb{E}[\|\vx^k-\hat{\vx}^{k+1}\|^2]+\frac{\hat \sigma_2}{2}\mathbb{E}[\|\vx^k-\vx^{k-1}\|^2]\\
\label{phiphi2}&~~~ -\left(\frac{1}{2\gamma}+\frac{3L+4\tau}{2(\gamma+2c+\kappa)^2}\right)\frac{N\sigma^2}{|\mathcal{I}|}.
\end{align}
Note that by using the definition of $\hat{\vz}^{k+1}$ in \eqref{hatz} and by \eqref{za}, we have
\begin{align}
  \hat \vz^{k+1}=  \vz^{k+1} + \beta ( \hat \vx^{k+1} -  \vx^{k+1}).
\end{align}
Thus, we can bound $\mathbb{E}[\|\vz^{k+1}-\vz^k\|^2]$ as
\begin{align*}
&\mathbb{E}[\|\vz^{k+1}-\vz^k\|^2]\\
&\geq(1-\frac{1}{\beta})\mathbb{E}[\|\vz^{k+1}-\hat{\vz}^{k+1}\|^2]+(1-\beta)\mathbb{E}[\|\hat{\vz}^{k+1}-\vz^k\|^2]\\
& =\beta(\beta-1)\mathbb{E}[\|\vx^{k+1}-\hat{\vx}^{k+1}\|^2]+(1-\beta)\mathbb{E}[\|\hat{\vz}^{k+1}-\vz^k\|^2]\\
&\geq \frac{\beta(\beta-1)}{(\gamma+2c+\kappa)^2}\frac{N\sigma^2}{|\mathcal{I}|}+(1-\beta)\beta^2\mathbb{E}
[\|\hat{\vx}^{k+1}-\vz^k\|^2],
\end{align*}
where the last inequality comes from  \eqref{xhatx} and \eqref{hatz}.
By substituting the above inequality  \eqref{phiphi2},  we obtain \eqref{phiphi4}.
\end{proof}
\subsection{Proof of Theorem \ref{The_conver}}
We are ready to prove Theorem \ref{The_conver}.
By summing \eqref{phiphi4} for $k=0,1,\ldots,K-1$, we obtain
\begin{align}
\nonumber&\mathbb{E}[\phi^0-\phi^{K}]\\
\nonumber&\geq  \frac{\kappa(1-\beta)\beta}{4}\sum_{k=0}^{K-1}\mathbb{E}[\|\hat{\vx}^{k+1}-\vz^k\|^2]-K\frac{C_1N\sigma^2}{|\mathcal{I}|}\\
\nonumber&~~~+\frac{\alpha}{2}\sum_{k=0}^{K-1}\mathbb{E}[\|\mA\vx(\vz^k,\vlambda^{k+1})\|^2]+\frac{\hat{\sigma}_1}{2}\sum_{k=0}^{K-1}
\mathbb{E}[\|\vx^k-\hat{\vx}^{k+1}\|^2]\\
 \label{phi0K}&~~~ +\frac{\hat{\sigma}_2}{2}\sum_{k=0}^{K-1}\mathbb{E}[\|\vx^k-\vx^{k-1}\|^2].
\end{align}

Recall the definition of $Q(\vx,\vlambda)$ in \eqref{QQ}
\begin{align}\label{QQ2}
\!\!\!Q(\vx,\vlambda)=&\|\vx-\textmd{prox}_{r}^1(\vx-\nabla f(\vx)-\mA^{\top}\vlambda)\|^2
+\|\mA\vx\|^2.
\end{align} To obtain the desired result, we first consider
\begin{align}
\nonumber&\mathbb{E}[\|\vx^k-\textmd{prox}_{r}^1(\vx^k-\nabla_{\vx}f(\vx^k)-\mA^{\top}\vlambda^{k+1})\|^2]\\
\label{gap1}&\leq 2\mathbb{E}[\|\vx^k-\hat{\vx}^{k+1}\|^2]\\
\nonumber&~~~+2\mathbb{E}[\|\hat{\vx}^{k+1}-\textmd{prox}_{r}^1(\vx^k-\nabla_{\vx}f(\vx^k)-\mA^{\top}\vlambda^{k+1})\|^2]
\end{align}
where the inequality dues to $\|a+b\|^2\leq2\|a\|^2+2\|b\|^2$. Notice
\begin{align}
\nonumber &\mathbb{E}[\|\hat{\vx}^{k+1}-\textmd{prox}_{r}^1(\vx^k-\nabla_{\vx}f(\vx^k)-\mA^{\top}\vlambda^{k+1})\|^2]\\
\nonumber&=\mathbb{E}[\|\textmd{prox}_{r}^1(\hat{\vx}^{k+1}-\nabla_{\vx} g(\hat{\vx}^{k+1},\vw^k;\vlambda^{k+1}))\\
\nonumber&~~~-\textmd{prox}_{r}^1(\vx^k-\nabla_{\vx}f(\vx^k)-\mA^{\top}\vlambda^{k+1})\|^2]\\
\nonumber &\leq \mathbb{E}[\|\hat{\vx}^{k+1}-\vx^k-\nabla_{\vx} g(\hat{\vx}^{k+1},\vw^k;\vlambda^{k+1})\\
\nonumber&~~~+\nabla_{\vx}f(\vx^k)+\mA^{\top}\vlambda^{k+1}\|^2]\\
\nonumber&\leq 2\mathbb{E}[\|\hat{\vx}^{k+1}-\vx^k\|^2]\\
\nonumber&~~~+2\mathbb{E}[\|\nabla_{\vx} g(\hat{\vx}^{k+1},\vw^k;\vlambda^{k+1})-\nabla_{\vx}f(\vx^k)-\mA^{\top}\vlambda^{k+1}\|^2],\\
\nonumber&=2\mathbb{E}[\|\hat{\vx}^{k+1}-\vx^k\|^2]+2\mathbb{E}[\|\nabla_{\vx}f(\vs^k)-\nabla_{\vx}f(\vx^k)\\
\nonumber&~~~+\gamma(\hat{\vx}^{k+1}-\vs^k)+c\mD(\hat{\vx}^{k+1}-\vx^k)+c\mA^T\mA\vx^k\\
\nonumber&~~~+\kappa(\hat{\vx}^{k+1}-\vz^k)\|^2]\\
\nonumber&\leq 2\mathbb{E}[\|\hat{\vx}^{k+1}-\vx^k\|^2]+10\mathbb{E}[\|\nabla_{\vx}f(\vs^k)-\nabla_{\vx}f(\vx^k)\|^2\\
\nonumber&~~~+\|\gamma(\hat{\vx}^{k+1}-\vs^k)\|^2+\|c\mD(\hat{\vx}^{k+1}-\vx^k)\|^2+\|c\mA^T\mA\vx^k\|^2\\
\nonumber&~~~+\|\kappa(\hat{\vx}^{k+1}-\vz^k)\|^2]\\
\nonumber&\leq (2+10c^2d_{max}^2+20\gamma^2)\mathbb{E}[\|\hat{\vx}^{k+1}-\vx^k\|^2]+10c^2\sigma_A^2\mathbb{E}[\|\mA\vx^k\|^2]\\
\nonumber&~~~+(10L^2+20\gamma^2)\mathbb{E}[\|\vx^k-\vs^k\|^2]+10\kappa^2\mathbb{E}[\|\hat{\vx}^{k+1}-\vz^k\|^2],
\end{align}
 where $d_{\max}$ is the largest value of matrix $\mD$, the first equality is due to the optimal condition for
\eqref{hatx}, i.e., $\hat{\vx}^{k+1}=\textmd{prox}_{r}^1(\hat{\vx}^{k+1}-\nabla_{\vx} g(\hat{\vx}^{k+1},\vw^k;\vlambda^{k+1})$; the first inequality is owing to the nonexpansive property
of the proximal operator; the second inequality dues to $\|a+b\|^2\leq2\|a\|^2+2\|b\|^2$; the second equality is obtained by the definition of function $g$ in \eqref{fun_g1}; %the third inequality is due to $\|\sum_{i=1}^5a_i\|^2\leq5\sum_{i=1}^5\|a_i\|^2$;
the last inequality dues to the $L$-smooth in \eqref{Lip}.

 Next, we show the upper bound for $\|\mA\vx^k\|$ as
\begin{align}
\nonumber&\mathbb{E}[\|\mA\vx^k\|^2]\\
\nonumber&\leq 2\mathbb{E}[\|\mA\vx(\vz^k,\vlambda^{k+1})\|^2]+2\sigma_A^2\mathbb{E}[\|\vx^k-\vx(\vz^k,\vlambda^{k+1})\|^2]\\
\nonumber&\leq 2\mathbb{E}[\|\mA\vx(\vz^k,\vlambda^{k+1})\|^2]+4\sigma_A^2\sigma_1^2\mathbb{E}[\|\vx^k-\hat{\vx}^{k+1}\|^2]\\
\label{gap2}&~~~+4\sigma_A^2{\sigma_2^2}\mathbb{E}[\|\vx^k-\vs^k\|^2],
\end{align}
where the last inequality comes from Lemma \ref{primal_error_bound}.
Now, we consider the upper bound of \eqref{QQ2}. Using the above inequalities \eqref{gap1}-\eqref{gap2}, we can obtain
\begin{align}
\nonumber & \min_{k=0,\ldots,K-1}\mathbb{E}[Q(\vx^{k},\vlambda^{k+1})]\\
\nonumber&\leq\frac{1}{K}\sum_{k=0}^{K-1}\mathbb{E}[\|\vx^k-\textmd{prox}_{r}^1(\vx^k-\nabla_{\vx}f(\vx^k)-\mA^{\top}\vlambda^{k+1})\|^2]\\
\nonumber&~~~+\frac{1}{K}\sum_{k=0}^{K-1}\mathbb{E}[\|\mA\vx^k\|^2]\\
\nonumber&\leq \frac{K_1}{K}\sum_{k=0}^{K-1}\mathbb{E}[\|\vx^k-\hat{\vx}^{k+1}\|^2+\frac{K_2}{K}\sum_{k=0}^{K-1}\|\vx^k-\vx^{k-1}\|^2]\\
\nonumber&~~~+\frac{K_3}{K}\sum_{k=0}^{K-1}\mathbb{E}[\|\hat{\vx}^{k+1}-\vz^k\|^2]+\frac{K_4}{K}\sum_{k=0}^{K-1}\mathbb{E}[\|\mA\vx(\vz^k,\vlambda^{k+1})\|^2].
\end{align}
where
\begin{align}
\nonumber K_1 &= 6+40\gamma+20c^2d_{\max}^2+4(20c^2\sigma_A^2+1)\sigma_A^2\sigma_1^2,\\
\nonumber K_2&= (20L^2+20\gamma^2) \bar{\eta}^2+4(20c^2\sigma_A^2+1)\sigma_A^2\sigma_2^2  \bar{\eta}^2,\\
\nonumber K_3&=20\kappa^2,~
\nonumber K_4=2(20c^2\sigma_A^2+1).
\end{align}
Further applying \eqref{phi0K}, we have
\begin{align}
\nonumber  \min_{k=0,\ldots,K-1}\mathbb{E}[Q(\vx^{k},\vlambda^{k+1})]& \leq C_0\left(\frac{\mathbb{E}[\phi^0-\phi^{K}]}{K}+\frac{C_1N\sigma^2}{|\mathcal{I}|}\right)\\
\nonumber & \leq C_0\left(\frac{\phi^0-\underline{f}}{K}+\frac{C_1N\sigma^2}{|\mathcal{I}|}\right),
\end{align}
where $\underline{f}$ is the lower bound of $\phi$ and $C_0$ is defined  as follows,
\begin{align}
C_0&\triangleq\frac{2K_1}{\hat{\sigma}_1}+\frac{K_2}{ \hat{\sigma}_2}
+\frac{4K_3\beta}{\kappa(1-\beta)}+\frac{2K_4}{\alpha}, \label{C0}
\end{align}

\section{Proof of Theorem \ref{thm PPDM}}\label{appendix thm2}
\begin{proof}
If we know the full gradient $\nabla f(\vx^k)$, i.e., $G(\vx^k,\vxi^k)=\nabla f(\vx^k)$, then $\sigma^2=0$.
Substituting it into \eqref{phiphi4}, we have
\begin{align}
\nonumber&\phi^k-\phi^{k+1}\\
\nonumber&\geq \frac{\kappa(1-\beta)\beta}{4}\|\hat{\vx}^{k+1}-\vz^k\|^2+\frac{\alpha}{2}\|\mA\vx(\vz^k,\vlambda^{k+1})\|^2\\
\nonumber&~~~+\frac{\hat{\sigma}_1}{2}\|\vx^k-\hat{\vx}^{k+1}\|^2+\frac{\hat{\sigma}_2}{2}\|\vx^k-\vx^{k-1}\|^2\geq0.
\end{align}
Thus $\phi^k$ is monotonically decreasing and it has lower bound $\underline{f}$. This implies that
\begin{align*}
 \max\{&\|\vx^k-\hat{\vx}^{k+1}\|,\|\vz^k-\hat{\vx}^{k+1}\|,\|\mA\vx(\vz^k;\vlambda^{k+1})\|\}\rightarrow 0.
\end{align*}
Thus, according to \cite[Theorem 2.4]{zhang2020}, every limit point generated by  PPDM algorithm is  a KKT point of  problem \eqref{eqn: dec_compact}.
In addition, substituting $\sigma^2=0$ into \eqref{QQQ1} and picking $K\geq \frac{C_0(\phi^0-\underline{f})}{\epsilon}$, we have
\begin{align}
\nonumber \min_{k=0,\ldots,K-1} Q(\vx^{k},\vlambda^{k+1}) \leq C_0\left(\frac{\phi^0-\underline{f}}{K}\right)\leq \epsilon.
\end{align}
Therefore, the proof is completed.
\end{proof}

%\bibliographystyle{ieeetr}
%\bibliography{wang_bib}

% biography section
%
% If you have an EPS/PDF photo (graphicx package needed) extra braces are
% needed around the contents of the optional argument to biography to prevent
% the LaTeX parser from getting confused when it sees the complicated
% \includegraphics command within an optional argument. (You could create
% your own custom macro containing the \includegraphics command to make things
% simpler here.)
%\begin{IEEEbiography}[{\includegraphics[width=1in,height=1.25in,clip,keepaspectratio]{mshell}}]{Michael Shell}
% or if you just want to reserve a space for a photo:

% You can push biographies down or up by placing
% a \vfill before or after them. The appropriate
% use of \vfill depends on what kind of text is
% on the last page and whether or not the columns
% are being equalized.

%\vfill

% Can be used to pull up biographies so that the bottom of the last one
% is flush with the other column.
%\enlargethispage{-5in}

% that's all folks
\end{document}